%% file: ViT-K-SIAM.tex
\documentclass[review,hidelinks,onefignum,onetabnum]{siamart251216}

\usepackage{booktabs}   
\usepackage{multirow}   
\usepackage{amsmath}
\usepackage{amssymb}

\usepackage{graphicx}   
\usepackage{subcaption} 
\usepackage{tikz}

\usetikzlibrary{
	shapes.geometric,   
	arrows.meta,        
	positioning,        
	shadows,            
	fit,                
	backgrounds         
}

\usepackage{xcolor} 

\usepackage{algorithm}
\nolinenumbers
\input{ex_shared}

\ifpdf
\hypersetup{
  pdftitle={ViT-K: A Few-Shot Learning Model for Coupled Fluid–Porous Media Flows with Interface Conditions},
  pdfauthor={M. Chen, C. Qiu, Z. Mao and M. Xu}
}
\fi

\begin{document}

\maketitle

\begin{abstract}
	The numerical simulation of interaction between free flow and porous media, governed by coupled Stokes/Navier--Stokes--Darcy flows, is critical for understanding fluid filtration and physiological transport, yet it is hindered by the high computational cost of resolving interface heterogeneities and the instability of long-term predictions. While deep learning offers surrogate modeling potential, existing frameworks often suffer from exponential error accumulation and poor convergence in multi-physics regimes. To address these limitations, we propose ViT-K, a novel few-shot learning model designed to learn the spatiotemporal evolution of coupled flows from sparse datasets. The ViT-K framework effectively reconstructs the global flow physics on a low-dimensional manifold by combining Vision Transformers (ViT) to capture heterogeneous interfacial features with the Koopman operator to linearize temporal dynamics. By lifting nonlinear dynamics into a globally linear observable space, the ViT-K model provides stability by design, ensuring that prediction errors grow linearly rather than exponentially over time. This theoretical property enables reliable long-term extrapolation even in small-sample regimes. Numerical experiments on benchmark coupled systems demonstrate that ViT-K not only captures complex interface physics with high fidelity but also exhibits exceptional robustness against measurement noise by acting as an implicit spectral filter. The proposed method significantly outperforms traditional solvers in inference speed while maintaining physical consistency, offering a robust paradigm for real-time multiphysics forecasting.
\end{abstract}



\section{Introduction}
\label{sec:intro}

Modeling the interaction between viscous free flow and porous media seepage governed by the coupled Stokes--Darcy or Navier--Stokes--Darcy (NSD) systems remains a central topic in computational mechanics. In the past decade, these models have been extensively investigated due to their critical role in diverse contexts, ranging from groundwater hydrology \cite{li2016fully,discacciati2004domain,hoppe2007computational} and karst aquifer systems \cite{cao2010finite,li2018stabilized} to petroleum recovery \cite{hanspal2006numerical,zhang2019fabrication} and biofluid transport \cite{zingaro2023comprehensive}. Owing to their inherent nonlinear coupling and multiscale nature, these systems have become canonical benchmarks for studying multiphysics flow interactions and for developing advanced numerical and data-driven modeling techniques. Significant efforts have been devoted to the numerical resolution of coupled systems, establishing a broad spectrum of robust solvers. Prominent among these are finite element and mixed finite element methods \cite{mardal2002robust,layton2002coupling,arbogast2007computational,camano2015new,gatica2009conforming,riviere2005locally,discacciati2007robin,qiu2020domain}, as well as finite difference and finite volume schemes \cite{hanspal2006numerical,li2018stabilized,lipnikov2014discontinuous}. Despite the mathematical rigor of these traditional approaches, they are inherently constrained by the computational burden of high-fidelity mesh generation. This bottleneck becomes critical in dynamic, long-duration simulations, motivating a shift towards efficient data-driven alternatives and surrogate models.

Amid the rapid advancement of scientific machine learning, physics-informed neural networks (PINNs) \cite{cai2021physics,Mao-2020} have emerged as a dominant paradigm for solving fluid dynamics problems. For coupled Stokes--Darcy systems, \cite{pu2022physics} employed monolithic PINNs with soft constraints to approximate global solutions. However, to mitigate the optimization pathology inherent in monolithic formulations dealing with multi-physics stiffness, the field has shifted toward domain-decomposition strategies. Notable developments include coupled neural-network frameworks \cite{yue2023efficient}, stochastic extensions \cite{li2025mc}, and hard-constraint parallel solvers \cite{lu2025high}, all of which significantly enhance numerical stability near complex interfaces. Complementing these instance-specific solvers, neural operators have emerged as a powerful paradigm for learning mesh-independent mappings of parametric PDEs. Foundational architectures, such as the Fourier neural operator (FNO) \cite{li2020fourier} and DeepONet \cite{lu2021learning}, along with advanced variants like U-NO \cite{rahman2022u} and PINO \cite{li2024physics}, have revolutionized the field. However, these operator-learning frameworks were not developed for solving coupled Navier–Stokes–Darcy systems with interface conditions effectively.

Despite these advances, the extension of deep learning paradigms to complex coupled systems remains an under-explored frontier. Specifically, the field confronts three fundamental hurdles: (i) \textbf{Training failure.} In regimes characterized by strong nonlinearities, PINNs frequently encounter convergence issues due to ill-conditioned loss landscapes \cite{krishnapriyan2021characterizing}. (ii) \textbf{Exponential accumulation of prediction errors.} Certain operator learning models are prone to severe error propagation during long-term temporal integration, where prediction errors grow exponentially when extrapolating beyond the training horizon \cite{takamoto2022pdebench, lippe2023pde}. (iii) \textbf{Prohibitive computational costs.} The substantial resource demands associated with training in large-scale and multi-scale domains continue to hinder practical engineering deployment \cite{wang2025time}.

To address the aforementioned challenges, we propose a novel data-driven framework, termed ViT-K, which synergistically integrates the spatial representation power of vision transformers (ViT) with the global linearization capabilities of the Koopman operator. As shown in Figure \ref{fig:vit-k-final}, on the spatial front, ViT excels at capturing global dependencies within high-dimensional flow fields. Unlike traditional convolutional architectures, its self-attention mechanism is particularly adept at learning the heterogeneous features across the fluid-porous interface in coupled NSD systems. Recent studies \cite{kang2023new,reddy2025twins} have demonstrated that ViT-based encoders can effectively extract global spatial modes. On the temporal front, the Koopman operator provides a rigorous theoretical basis for linearizing nonlinear dynamical systems—a critical attribute for long-term stability. By lifting the nonlinear dynamics into an infinite-dimensional observable space, the Koopman operator effectively mitigates error propagation issues \cite{brunton2021modern}. Crucially, this global linearization imparts intrinsic stability by design, empowering the ViT-K model to achieve reliable long-term extrapolation even in small-sample regimes. Furthermore, the framework acts as an implicit spectral filter, effectively attenuating stochastic perturbations and maintaining high predictive stability even when observational data is corrupted.

Consequently, the ViT-K framework facilitates robust long-term forecasting, offering a distinct dual advantage: it synergistically leverages ViT to resolve complex spatial correlations with high fidelity, while employing Koopman dynamics to ensure stable temporal extrapolation. By transforming the nonlinear evolution into a stable, linear trajectory within a latent space, this approach not only circumvents the prohibitive computational overhead associated with explicit PDE solvers but also significantly accelerates numerical inference.

The rest of this paper is organized as follows. In \cref{sec:model}, we present the coupled Stokes--Darcy and Navier--Stokes--Darcy models. \Cref{sec:vit_k_framework} introduces the ViT-K framework, while the structured Koopman formulation and its stability properties are analyzed in \cref{sec:koopman_theory}. \Cref{sec:numerical} provides numerical experiments demonstrating the accuracy and long-term stability of the proposed method. Finally, \cref{sec:conclusion} concludes the paper.

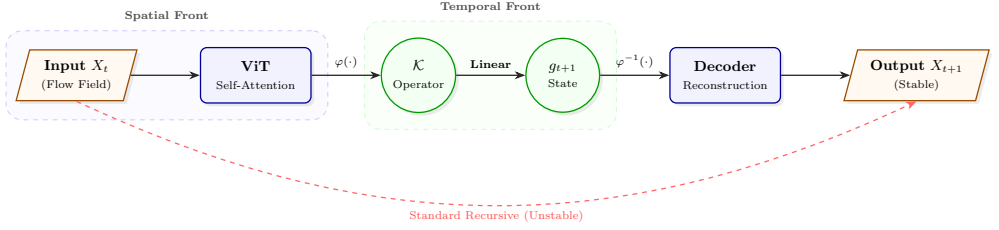
\begin{figure}[h]
	\centering
	\resizebox{\textwidth}{!}{%
		\begin{tikzpicture}[
			node distance=1.2cm and 1.5cm, 
			font=\sffamily,
			process/.style={
				rectangle, 
				draw=blue!60!black, 
				fill=blue!5, 
				thick, 
				rounded corners, 
				minimum height=1.2cm, 
				minimum width=2.4cm, 
				align=center, 
				drop shadow={opacity=0.15, shadow xshift=2pt, shadow yshift=-2pt}, 
				font=\small
			},
			operator/.style={
				circle, 
				draw=green!60!black, 
				fill=green!5, 
				thick, 
				minimum size=1.6cm, 
				align=center, 
				drop shadow={opacity=0.15, shadow xshift=2pt, shadow yshift=-2pt},
				font=\small
			},
			data/.style={
				trapezium, 
				trapezium left angle=75, 
				trapezium right angle=105, 
				draw=orange!60!black, 
				fill=orange!5, 
				thick, 
				minimum height=1.1cm, 
				inner sep=5pt, 
				align=center, 
				drop shadow={opacity=0.15, shadow xshift=2pt, shadow yshift=-2pt},
				font=\small
			},
			arrow_label/.style={
				midway,          
				above,           
				font=\scriptsize\bfseries, 
				text=black,
				yshift=1pt       
			},
			arrow/.style={-Stealth, thick, color=gray!30!black},
			group_title/.style={
				font=\bfseries\scriptsize, 
				color=gray!60!black,
				yshift=2pt 
			}
			]
			
			\node[data] (input) {\textbf{Input} $X_t$\\ \scriptsize (Flow Field)};
			\node[process, right=of input] (vit) {\textbf{ViT}\\ \scriptsize Self-Attention};
			\node[operator, right=of vit] (koopman) {$\mathcal{K}$\\ \scriptsize Operator};
			\node[operator, right=of koopman] (next_latent) {$g_{t+1}$\\ \scriptsize State};
			\node[process, right=of next_latent] (decoder) {\textbf{Decoder}\\ \scriptsize Reconstruction};
			\node[data, right=of decoder] (output) {\textbf{Output} $X_{t+1}$\\ \scriptsize (Stable)};
			
			\draw[arrow] (input) -- (vit);
			
			\draw[arrow] (vit) -- node[arrow_label] {$\varphi(\cdot)$} (koopman);
			
			\draw[arrow] (koopman) -- node[arrow_label] {Linear} (next_latent);
			
			\draw[arrow] (next_latent) -- node[arrow_label] {$\varphi^{-1}(\cdot)$} (decoder);
			
			\draw[arrow] (decoder) -- (output);
			
			\begin{scope}[on background layer]
				\node[fit=(input)(vit), 
				draw=blue!20, 
				fill=blue!2, 
				dashed, 
				rounded corners=10pt, 
				inner sep=10pt] (spatial_box) {};
				\node[group_title, above] at (spatial_box.north) {Spatial Front};
				
				\node[fit=(koopman)(next_latent), 
				draw=green!20, 
				fill=green!2, 
				dashed, 
				rounded corners=10pt, 
				inner sep=10pt] (temporal_box) {};
				\node[group_title, above] at (temporal_box.north) {Temporal Front};
			\end{scope}
			
			\draw[arrow, dashed, red!60, thick] (input.south) to[out=-25, in=205] 
			node[below, align=center, font=\scriptsize, text=red!70, fill=white, inner sep=2pt, yshift=-1pt] 
			{Standard Recursive (Unstable)} (output.south);
			
		\end{tikzpicture}%
	}
	\caption{The ViT-K Framework Architecture.}
	\label{fig:vit-k-final}
	\vspace{-2em}
\end{figure}

\section{Model Problems}
\label{sec:model}


Consider a bounded domain $\Omega \subset \mathbb{R}^d$ ($d \in \{2,3\}$) partitioned into a porous medium region $\Omega_D$ and a free-flow fluid region $\Omega_S$, separated by a common interface $\Gamma = \overline{\Omega}_D \cap \overline{\Omega}_S$. Let $\vec{n}_S$ and $\vec{n}_D$ denote the unit outward normal vectors on $\Gamma$ for $\Omega_S$ and $\Omega_D$, respectively.

\begin{figure}[htp]
	\centering
	\includegraphics[width=0.85\linewidth]{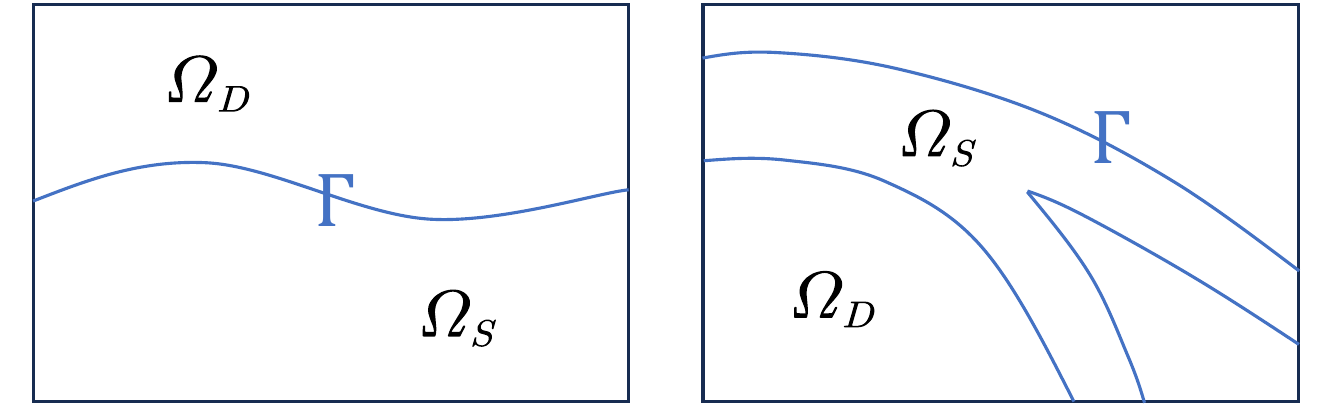}
	\caption{Schematic representation of the computational domain $\Omega$, consisting of the porous medium $\Omega_D$ and the fluid region $\Omega_S$ separated by the interface $\Gamma$.}
	\label{fig:domain}
	\vspace{-2em}
\end{figure}

In the porous medium $\Omega_D$, the flow dynamics are governed by mass conservation and Darcy's law. The hydraulic head $\phi_D$ and the filtration velocity $\vec{u}_D$ satisfy
\begin{subequations}
	\begin{align}
		S_0 \frac{\partial \phi_D}{\partial t} + \nabla \cdot \vec{u}_D &= f_D \quad &\text{in } \Omega_D \times (0, T], \label{eq:mass_cons} \\
		\vec{u}_D &= -\mathbb{K} \nabla \phi_D \quad &\text{in } \Omega_D \times (0, T], \label{eq:darcy_law}
	\end{align}
\end{subequations}
where $S_0$ is the specific storage coefficient, $\mathbb{K}$ is the hydraulic conductivity tensor (assumed isotropic, $\mathbb{K} = k\mathbb{I}$), and $f_D$ represents the source/sink term. The hydraulic head is defined as $\phi_D = z + p_D/(\rho g)$, with $p_D$ being the dynamic pressure. Substituting \eqref{eq:darcy_law} into \eqref{eq:mass_cons} yields the parabolic formulation
\begin{equation}
	S_0 \frac{\partial \phi_D}{\partial t} - \nabla \cdot (\mathbb{K} \nabla \phi_D) = f_D. \label{eq:darcy_second_order}
\end{equation}

In the free-flow region $\Omega_S$, the fluid velocity $\vec{u}_S$ and kinematic pressure $p_S$ are governed by the time-dependent Stokes equations:
\begin{subequations}
	\begin{align}
		\frac{\partial \vec{u}_S}{\partial t} - \nabla \cdot \mathbb{T}(\vec{u}_S, p_S) &= \vec{f}_S \quad &\text{in } \Omega_S \times (0, T], \label{eq:stokes_momentum} \\
		\nabla \cdot \vec{u}_S &= 0 \quad &\text{in } \Omega_S \times (0, T]. \label{eq:stokes_continuity}
	\end{align}
\end{subequations}
Here, $\mathbb{T}(\vec{u}_S, p_S) = 2\nu \mathbb{D}(\vec{u}_S) - p_S\mathbb{I}$ denotes the kinematic fluid stress tensor, where $\nu$ is the kinematic viscosity, and $\mathbb{D}(\vec{u}_S) = \frac{1}{2}(\nabla \vec{u}_S + (\nabla \vec{u}_S)^{\mathsf{T}})$ is the rate of strain tensor.

To close the coupled system, we impose physical coupling conditions on the interface $\Gamma$. These include the continuity of normal flux, the balance of normal forces, and a Beavers--Joseph condition:
\begin{subequations}
	\label{eq:interface_conditions}
	\begin{align}
		\vec{u}_S \cdot \vec{n}_S + \vec{u}_D \cdot \vec{n}_D &= 0, ~~
		-\vec{n}_S \cdot \mathbb{T}(\vec{u}_S, p_S) \cdot \vec{n}_S = g(\phi_D - z), \label{eq:normal_stress} \\
		-\vec{\tau}_j \cdot (\mathbb{T}(\vec{u}_S, p_S) \cdot \vec{n}_S) &= \frac{\alpha\nu\sqrt{d}}{\sqrt{\text{trace}({\Pi})}} \vec{\tau}_j \cdot (\vec{u}_S - \vec{u}_D), \quad j=1,\dots,d-1, \label{eq:full_BJ}
	\end{align}
\end{subequations}
where $g$ is the gravitational acceleration and $\{\vec{\tau}_j\}$ form an orthonormal basis for the tangent plane on $\Gamma$. $\alpha$ denotes the dimensionless Beavers--Joseph slip coefficient and ${\Pi} = (\mathbb{K}\nu)/g$ is the intrinsic permeability tensor. This condition accounts for the shear stress jump induced by the porous structure at the interface. Note that condition \eqref{eq:full_BJ} can be degenerated to a simplified Beavers--Joseph condition $\vec{u}_S \cdot \vec{\tau}_j = 0$ (effectively a no-slip condition).

Regarding boundary conditions on $\partial\Omega \setminus \Gamma$, we consider standard Dirichlet conditions. For the validation of numerical schemes using manufactured solutions, inhomogeneous Dirichlet data consistent with the exact solution are imposed.

In the Navier--Stokes--Darcy Model, the Stokes system in $\Omega_S$ is replaced by the Navier--Stokes equations to account for inertial effects in flows with higher Reynolds numbers. 
\begin{subequations}
	\begin{align}
		\frac{\partial \vec{u}_S}{\partial t} + (\vec{u}_S \cdot \nabla)\vec{u}_S - \nabla \cdot \mathbb{T}(\vec{u}_S, p_S) &= \vec{f}_S, \quad &\text{in } \Omega_S \times (0, T], \label{eq:N-stokes_momentum} \\
		\nabla \cdot \vec{u}_S &= 0 \quad &\text{in } \Omega_S \times (0, T]. \label{eq:N-stokes_continuity}
	\end{align}
\end{subequations}
The governing equations in the porous medium $\Omega_D$ remain identical to \eqref{eq:darcy_second_order}.

\section{The ViT-K Framework: Vision Transformer with Koopman Operator}\label{sec:vit_k_framework}

This section details the proposed ViT-K framework, a hybrid architecture that combines the spatial feature extraction capabilities of vision transformers (ViT) with the temporal extrapolation power of Koopman operator theory. By integrating these paradigms, the ViT-K is designed to effectively predict the evolving states of coupled Stokes--Darcy and Navier--Stokes--Darcy systems.

\subsection{Spatial Encoding via Vision Transformer}
\label{subsec:vit_structure}

The vision transformer, originally introduced by \cite{dosovitskiy2020image}, adapts the transformer architecture—traditionally used for natural language processing—to visual data. Unlike convolutional neural networks (CNNs), which rely on local receptive fields, ViT leverages a multi-head self-attention mechanism to capture global spatial dependencies \cite{touvron2021training,vaswani2017attention}. This characteristic is particularly advantageous for multiphysics modeling, as it allows for the effective capture of long-range interactions inherent in coupled fluid--porous media fields.

Consider a snapshot of the flow field represented as a 2D input tensor $\mathbf{X} \in \mathbb{R}^{H \times W \times C}$, where $H, W$ denote the spatial dimensions and $C$ represents the number of physical variables (channels). The typical ViT framework is shown in Figure \ref{fig:ViT}. The ViT encoder processes this input through the following steps:

\textbf{Patch Embedding:} The input $\mathbf{X}$ is partitioned into a sequence of $N$ flattened patches $\mathbf{X}_p \in \mathbb{R}^{N \times (P^2 \cdot C)}$, where $P$ is the patch size and $N = (HW)/P^2$ is the number of patches (assuming $H, W$ are divisible by $P$). These patches are linearly projected into a latent dimension $D$ via a learnable embedding matrix $\mathbf{E} \in \mathbb{R}^{(P^2 \cdot C) \times D}$.

\textbf{Positional Encoding and Tokenization:} To preserve spatial topology, learnable positional embeddings $\mathbf{E}_{pos}$ are added. Following the standard ViT protocol, a learnable \texttt{[class]} token $\mathbf{x}_{class} \in \mathbb{R}^{1 \times D}$ is prepended to the sequence. This token serves as a global aggregator of information. The initial input sequence $\mathbf{Z}_0 \in \mathbb{R}^{(N+1) \times D}$ is formed as $
	\mathbf{Z}_0 = \left[ \mathbf{x}_{class}; \, \mathbf{x}_p^1 \mathbf{E}; \, \dots; \, \mathbf{x}_p^N \mathbf{E} \right] + \mathbf{E}_{pos}$,
where the bracket notation $[\cdot ; \cdot]$ denotes row-wise concatenation.

\textbf{Transformer Encoder:} The sequence $\mathbf{Z}_0$ passes through $L$ Transformer layers. Each layer consists of Multi-Head Self-Attention (MSA) and a Feed-Forward Network (FFN), stabilized by Layer Normalization (LN) and residual connections:
\begin{align}
	\mathbf{Z}'_l &= \text{MSA}(\text{LN}(\mathbf{Z}_{l-1})) + \mathbf{Z}_{l-1},\qquad
	\mathbf{Z}_l &= \text{FFN}(\text{LN}(\mathbf{Z}'_l)) + \mathbf{Z}'_l, \quad l = 1, \dots, L.
\end{align}
The final output $\mathbf{Z}_L$ contains the encoded representation of the flow field. In our framework, the aggregated feature corresponding to the \texttt{[class]} token, $\mathbf{g} = \mathbf{Z}_L^{(0)} \in \mathbb{R}^{d}$, serves as the compact latent state vector for the subsequent Koopman analysis.

For context, in the canonical classification setting (which we adapt for regression), the output state of the \texttt{[class]} token $\mathbf{z}_L^0$ serves as the global representation. A Multi-Layer Perceptron (MLP) head maps this representation to logits $\mathbf{h} = \text{MLP}(\text{LN}(\mathbf{z}_L^0)) \in \mathbb{R}^K$. The final class probabilities are obtained via the Softmax function. The model is typically trained by minimizing the Cross-Entropy loss:
\begin{equation}
	\mathcal{L}_{\text{CE}} = -\frac{1}{N_{\text{batch}}} \sum_{i=1}^{N_{\text{batch}}} \sum_{k=1}^{K} y_{i,k} \log(\hat{y}_{i,k}),
\end{equation}
where $y_{i,k}$ is the one-hot encoded ground-truth label for the $i$-th sample, and $\hat{y}_{i,k}$ is the predicted probability. $K$ is the number of classes. While this architecture has achieved remarkable success in discrete classification, our ViT-K framework replaces this classification head with a Koopman-based dynamical evolution module, as detailed in Section \ref{subsec:vit_k_architecture}.

\begin{algorithm}[t]
	\footnotesize
	\caption{Standard Vision Transformer (ViT)}
	\label{alg:vit}
	\begin{algorithmic} 
		\REQUIRE Input image $\mathbb{X} \in \mathbb{R}^{H \times W \times C}$; Patch size $P$; embedding dimension $D$; Number of layers $L$.
		\ENSURE Class prediction $\boldsymbol{y} \in \mathbb{R}^{K}$.
		
		\STATE \textbf{Phase 1: Patch Partition \& Embedding}
		\STATE Divide $\mathbb{X}$ into $N = HW/P^2$ non-overlapping patches,
		\STATE Flatten patches to $\boldsymbol{x}_p \in \mathbb{R}^{N \times (P^2 C)}$ and project: $\mathbb{Z}_p = \boldsymbol{x}_p \mathbb{E}$,
		
		\STATE \textbf{Phase 2: Tokenization \& Position Encoding}
		\STATE Initialize learnable class token $\boldsymbol{x}_{\mathrm{class}} \in \mathbb{R}^{1 \times D}$,
		\STATE $\mathbb{Z}_0 = [\boldsymbol{x}_{\mathrm{class}}; \mathbb{Z}_p] + \mathbb{E}_{\mathrm{pos}}$, where $\mathbb{E}_{\mathrm{pos}} \in \mathbb{R}^{(N+1) \times D}$,
		
		\STATE \textbf{Phase 3: Transformer Encoder}
		\FOR{$l = 1$ \TO $L$}
		\STATE $\mathbb{Z}'_l = \mathrm{MSA}(\mathrm{LN}(\mathbb{Z}_{l-1})) + \mathbb{Z}_{l-1}$ \COMMENT{Multi-head Self-Attention},
		\STATE $\mathbb{Z}_l = \mathrm{FFN}(\mathrm{LN}(\mathbb{Z}'_l)) + \mathbb{Z}'_l$ \COMMENT{Feed-forward Network},
		\ENDFOR
		
		\STATE \textbf{Phase 4: Classification Head}
		\STATE Extract global representation: $\boldsymbol{z}_L^0 = \mathbb{Z}_L$,
		\STATE Compute prediction: $\boldsymbol{y} = \mathrm{Linear}(\mathrm{LN}(\boldsymbol{z}_L^0))$,
		\RETURN $\boldsymbol{y}$.
	\end{algorithmic}
\end{algorithm}

\begin{figure}[htp]
	\centering
	\includegraphics[width=0.6\linewidth]{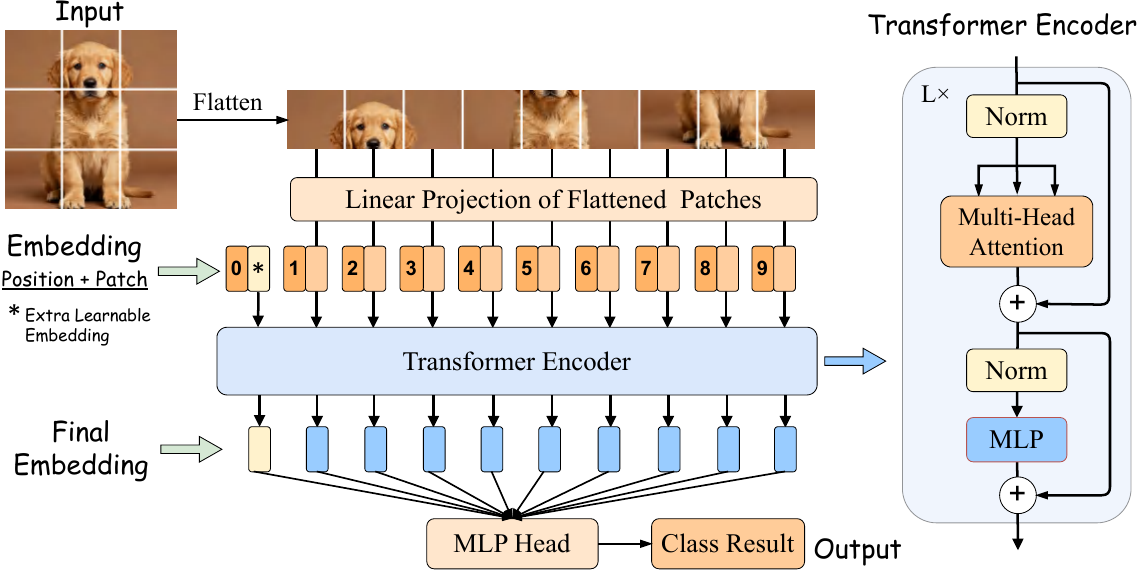} 
	\caption{Schematic of the standard Vision Transformer architecture.}
	\label{fig:ViT}
	\vspace{-3em}
\end{figure}

\subsection{Temporal Evolution via Koopman Operator}
\label{subsec:koopman_operator}

For continuous-time dynamical systems governing the Stokes-Darcy or Navier-Stokes-Darcy flows, the state evolution $x(t+\Delta t) = \Phi_{\Delta t}(x(t))$ is inherently nonlinear. This nonlinearity poses significant challenges for long-horizon extrapolation directly in the physical state space.

The Koopman operator theory~\cite{koopman1931hamiltonian,mezic2005spectral} circumvents this by lifting the dynamics into an infinite-dimensional space of observables, where the evolution becomes linear. Let $\mathcal{X}$ be the state space and $\mathcal{G}$ be a Banach space of scalar observables $g: \mathcal{X} \to \mathbb{C}$. The Koopman operator $\mathcal{K}_{\Delta t}$ acts on observables as $(\mathcal{K}_{\Delta t} g)(x) = g(\Phi_{\Delta t}(x))$.
Its continuous-time infinitesimal generator $\mathcal{L}$ satisfies:
\begin{equation}
	\frac{d}{dt} g(\Phi_t(x)) = (\mathcal{L}g)(\Phi_t(x)), \quad \text{with} \quad \mathcal{K}_{\Delta t} = e^{\Delta t \mathcal{L}}.
\end{equation}
While $\Phi_t$ is nonlinear in $\mathcal{X}$, both $\mathcal{K}_{\Delta t}$ and $\mathcal{L}$ act linearly on observables $g\in\mathcal{G}$. Crucially, if an observable lies in the span of the eigenfunctions of $\mathcal{L}$ with eigenvalue $\omega$, i.e., $\mathcal{L}\phi = \omega\phi$, its time evolution is simply $\phi(\Phi_t(x)) = e^{\omega t}\phi(x)$. This implies that highly complex nonlinear dynamics can be decomposed into a superposition of exponential modes in the observable space.

\subsection{ViT-K Architecture for Coupled Systems}
\label{subsec:vit_k_architecture}
While standard Vision Transformers (ViT) excel at spatial feature extraction, their discrete classification architecture is ill-suited for continuous fluid dynamics regression. To overcome the numerical instability and error accumulation inherent in simulating coupled Stokes/Navier-Stokes-Darcy (NSD) systems, we propose the ViT-K framework. By discarding traditional classification components (e.g., Softmax and Cross-Entropy loss), we repurpose the ViT encoder as a lifting function that maps nonlinear spatio-temporal dynamics into a Koopman-invariant observable space. This transformation into a globally linear latent space governed by a structured operator ensures stable and interpretable predictions of coupled fluid evolution.

The overall architecture is illustrated in Figure \ref{fig:ViT-K}. The algorithm of ViT-K are shown in Algorithm \ref{alg:vit-k}. The framework consists of three key modules:
\begin{figure}
	\centering
	\includegraphics[width=0.8\linewidth]{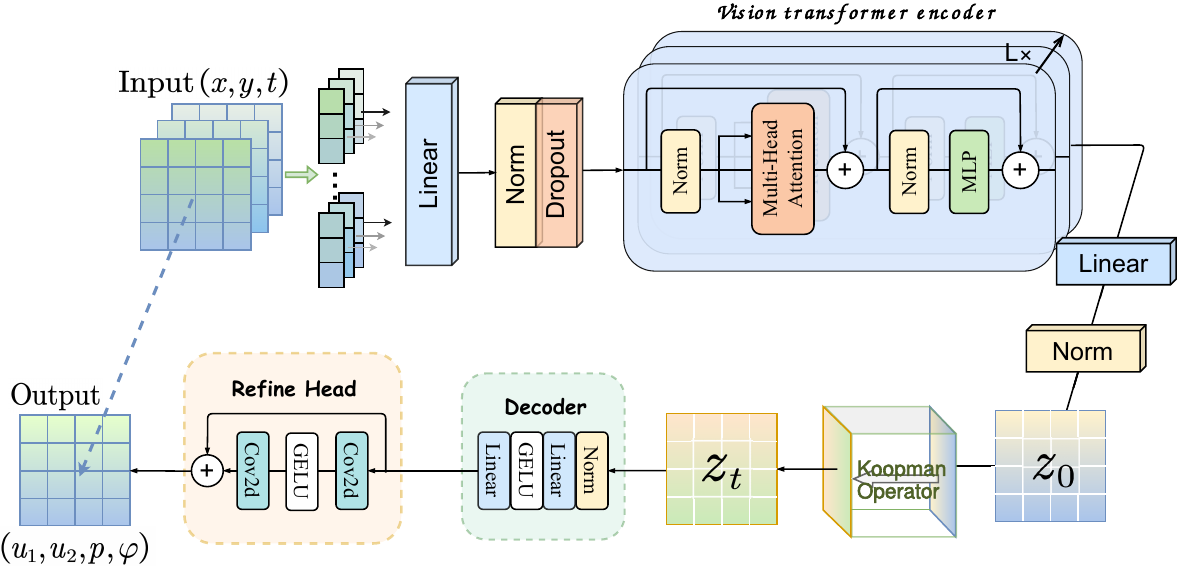}
	\caption{Architecture of the proposed ViT-K framework. The ViT encoder extracts spatial observables from NSD snapshots, and the Koopman operator advances these observables linearly in time.}
	\label{fig:ViT-K}
	\vspace{-3em}
\end{figure}

\textbf{(1) Spatial Encoding (Modified ViT):} 
	We utilize the backbone of the ViT encoder described in Section \ref{subsec:vit_structure}. The input flow fields (velocity $\vec{u}$, pressure $p$, potential $\phi$) are patch-embedded and processed by the Transformer layers. However, instead of feeding the final \texttt{[class]} token $\mathbf{z}_L^0$ into a classifier, we treat it as the observable vector $\mathbf{g}(t) \in \mathbb{R}^d$ in the Koopman analysis. The ViT Encoder acts as the lifting function $\Psi_{enc}$, mapping the physical fields (velocity, pressure, potential) into a compact latent vector $\mathbf{g} \in \mathbb{R}^d$. The multi-head attention mechanism ensures that global coupling effects between $\Omega_S$ and $\Omega_D$ are encoded efficiently.

\textbf{(2) Temporal Evolution (Koopman Layer):} A learnable finite-dimensional linear operator $\mathbf{A} \approx \mathcal{L}$ advances the latent state: $\mathbf{g}(t+\Delta t) = e^{\mathbf{A}\Delta t}\mathbf{g}(t)$ which replaces the discrete Softmax output with a continuous temporal trajectory. This operator advances the latent states linearly in time, replacing the nonlinear dynamics of the original PDE system. Such linear evolution ensures stability and interpretability and allows extrapolation over long temporal horizons.

\textbf{(3) Physical Reconstruction (Decoder and Refinement):} A reconstruction network $\Psi_{dec}$ maps the evolved latent states $\mathbf{g}(t+\Delta t)$ back to the physical domain, recovering the full fields $\hat{\mathbf{u}}, \hat{p}, \hat{\phi}$. A combination of linear projection and local convolutional refinement, as shown in Refine Head in the Figure \ref{fig:ViT-K}, ensures consistency between global structures and fine spatial details.

\begin{algorithm}[t]
	\caption{Algorithm of ViT-K}
	\label{alg:vit-k}
	\begin{algorithmic} 
		\REQUIRE Discrete initial field $\mathbb{U}_h(\boldsymbol{x}, t_0) \in \mathbb{R}^{C}$ on domain $\Omega$; Spatial grid coordinates $\mathbb{X}_{\text{grid}}$; Target time sequence $\mathcal{T}=\{t_1, \dots, t_M\}$.
		\ENSURE Approximated solution trajectory $\{\hat{\mathbb{U}}_h(\boldsymbol{x}, t_k)\}_{k=1}^M$.
		
		\STATE \textbf{Phase 1: Spatial Encoding (Modified ViT)}
		\STATE \textbf{Harmonic Embedding:} Augment the initial condition with positional encodings $\gamma(\mathbb{X}_{\text{grid}})$ to capture multi-scale spatial frequencies.
		\STATE \hspace*{\algorithmicindent} $\mathbb{I}_0 \leftarrow \mathrm{Concat}\left(\mathbb{U}_h(\boldsymbol{x}, t_0), \gamma(\mathbb{X}_{\text{grid}})\right)$
		
		\STATE \textbf{Global Encoding:} Map the augmented input $\mathbb{I}_0$ into the Koopman invariant subspace $\mathcal{H} \cong \mathbb{R}^{d}$ using the lifting function $\Psi_{\text{enc}}$.
		\STATE \hspace*{\algorithmicindent} $\boldsymbol{g}(t_0) \leftarrow \Psi_{\text{enc}}(\mathrm{PatchEmbed}(\mathbb{I}_0)) \in \mathbb{R}^{d}$
		
		\STATE \textbf{Phase 2: Temporal Evolution (Koopman Layer)}
		\FOR{\textbf{each} target time $t_k \in \mathcal{T}$}
		\STATE Compute time horizon $\tau_k = t_k - t_0$.
		\STATE \textbf{Spectral Evolution:} Advance latent observable via matrix exponential of linear operator $\mathbb{A}$ (approximating Koopman operator $\mathcal{K}^{\tau_k}$).
		\STATE \hspace*{\algorithmicindent} $\boldsymbol{g}(t_k) \leftarrow \exp\left(\mathbb{A} \cdot \tau_k\right) \boldsymbol{g}(t_0)$,
		\STATE \textit{Note: $\mathbb{A}$ models coupled oscillatory and dissipative dynamics.}
		\ENDFOR
		
		\STATE \textbf{Phase 3: Physical Reconstruction}
		\FOR{\textbf{each} evolved latent state $\boldsymbol{g}(t_k)$}
		\STATE \textbf{Decoding:} Project $\boldsymbol{g}(t_k)$ back to patch space via $\Psi_{\text{dec}}$ for coarse approximation $\tilde{\mathbb{U}}$.
		\STATE \textbf{Residual Refinement:} Apply convolutional operator $\Psi_{\text{ref}}$ to recover interface details.
		\STATE \hspace*{\algorithmicindent} $\hat{\mathbb{U}}_h(\boldsymbol{x}, t_k) \leftarrow \Psi_{\text{ref}}\left( \mathrm{Reshape}(\Psi_{\text{dec}}(\boldsymbol{g}(t_k))) \right)$,
		\ENDFOR
		
		\STATE \textbf{Training Optimization (if applicable)}
		\IF{training mode}
		\STATE Minimize domain-weighted functional $\mathcal{J}$ over domains $\Omega_S$ and $\Omega_D$:
		\STATE \hspace*{\algorithmicindent} $\mathcal{J}(\theta) = \sum_{k=1}^{M} \left[ \sum_{v \in \{\boldsymbol{u}, p\}} w_v \| \hat{v} - v \|_{\Omega_S}^2 + w_{\phi} \| \hat{\phi} - \phi \|_{\Omega_D}^2 \right]$,
		\STATE Update parameters $\theta$ via backpropagation.
		\ENDIF

		\RETURN Predicted fields $\{\hat{\mathbb{U}}_h(\boldsymbol{x}, t_k)\}_{k=1}^M$.
	\end{algorithmic}
\end{algorithm}

Since the task is no longer classification, the cross-entropy loss ($\mathcal{L}_{\text{CE}}$) is replaced by a physics-aware regression objective. The total loss function minimizes the reconstruction error over the fluid ($\Omega_S$) and porous ($\Omega_D$) subdomains, combined with a linearity constraint for the Koopman operator. To ensure physical consistency across the heterogeneous domains, we employ a domain-weighted reconstruction loss. In the fluid subdomain $\Omega_S$ and the porous subdomain $\Omega_D$, which have grid counts $N_{\Omega_S}$ and $N_{\Omega_D}$, respectively, the total training objective is:
\begin{equation}
	\begin{split}
		\mathcal{L}_{total} &= w_{u_1}\mathrm{MSE}_{\Omega_S}(\hat{u}_1, u_1) + w_{u_2}\mathrm{MSE}_{\Omega_S}(\hat{u}_2, u_2) \\
		&\quad + w_{p}\mathrm{MSE}_{\Omega_S}(\hat{p}, p) + w_{\phi}\mathrm{MSE}_{\Omega_D}(\hat{\phi}, \phi) + \lambda \mathcal{L}_{linearity},
	\end{split}
\end{equation}
where the domain-specific MSE is defined as $\mathrm{MSE}_{\Omega}(\hat{f}, f) = \frac{1}{N_{\Omega}} \sum_{\vec{x} \in \Omega} \|\hat{f}(\vec{x}) - f(\vec{x})\|^2$. The term $\mathcal{L}_{linearity}$ enforces the linear dynamics in the latent space and
\begin{equation}
	\mathcal{L}_{\text{linearity}} = \frac{1}{M} \sum_{k=1}^{M} \left\| \Psi_{\text{enc}}(\vec{x}_{k+1}) - \mathrm{e}^{\Delta t_k \mathbb{A}} \Psi_{\text{enc}}(\vec{x}_k) \right\|^2.
	\label{eq:linearity_loss-1}
\end{equation}
This composite loss ensures balanced accuracy across the multiscale interface. This loss computes weighted mean-square errors in each subdomain, ensuring balanced accuracy and physical consistency in both fluid and porous regions. Minimizing $\mathcal{L}_{total}$ jointly optimizes the encoder, the linear operator, and the decoder, achieving unified modeling of NSD dynamics and stable temporal extrapolation.

Overall, ViT-K combines the global spatial representation power of Transformers with the linear and interpretable dynamics of the Koopman framework, enabling accurate, stable, and physically consistent long-term predictions for coupled flow systems.

\section{Structured Koopman Operator: Stability and Error Analysis}
\label{sec:koopman_theory}

Ensuring long-term stability and physical plausibility is a central challenge in data-driven modeling of dynamical systems. Standard approaches that learn a dense, unconstrained Koopman generator often suffer from spectral instability, leading to overfitting of spurious growing modes. To overcome these limitations, we introduce a rigorous \textit{inductive bias} by imposing a physically-motivated structure on the Koopman generator $\mathbb{A}$. Our design is guided by the principle that viscous fluid systems are inherently dissipative and should not spontaneously generate energy.

\subsection{Dissipative Structure via Operator Decomposition}
\label{subsec:general_stable_structure}
Any real matrix $\mathbf{A} \in \mathbb{R}^{d \times d}$ uniquely decomposes as $\mathbf{A} = \mathbf{S} + \mathbf{W}$, where $\mathbf{S}$ is symmetric and $\mathbf{W}$ is skew-symmetric. To guarantee stable, non-expansive learned dynamics, we explicitly constrain $\mathbf{S}$ to be negative semi-definite ($\mathbf{S} \preceq 0$). From a physical perspective, $\mathbf{S} \preceq 0$ governs monotonic energy dissipation, while the conservative component $\mathbf{W}$ captures energy-preserving oscillations and convection.

%

\begin{proposition}[Lyapunov Stability of the Generator]
	\label{prop:general_stability}
	Let the Koopman generator $\mathbf{A} = \mathbf{S} + \mathbf{W}$, where $\mathbf{S} \preceq 0$ is symmetric negative semi-definite and $\mathbf{W}^\top = -\mathbf{W}$ is skew-symmetric. The continuous-time evolution operator $\mathcal{K}_t = e^{t\mathbf{A}}$ is a contraction in the Euclidean norm for all $t \ge 0$, i.e., $\|\mathcal{K}_t\|_2 \le 1$.
\end{proposition}

\begin{proof}
	Consider the Lyapunov function $V(\mathbf{z}) = \frac{1}{2}\|\mathbf{z}\|^2$ for a state $\mathbf{z}$ evolving according to $\dot{\mathbf{z}} = \mathbf{A}\mathbf{z}$. The time derivative is:
	\begin{equation}
		\frac{d}{dt}V(\mathbf{z}) = \mathbf{z}^\top \mathbf{A} \mathbf{z} = \mathbf{z}^\top (\mathbf{S} + \mathbf{W}) \mathbf{z} = \mathbf{z}^\top \mathbf{S} \mathbf{z} + \mathbf{z}^\top \mathbf{W} \mathbf{z}.
	\end{equation}
	Since $\mathbf{W}$ is skew-symmetric, $\mathbf{z}^\top \mathbf{W} \mathbf{z} = 0$. Since $\mathbf{S} \preceq 0$, we have $\dot{V}(\mathbf{z}) = \mathbf{z}^\top \mathbf{S} \mathbf{z} \le 0$. Thus, the energy is non-increasing, implying $\|\mathbf{z}(t)\| \le \|\mathbf{z}(0)\|$ for all $t \ge 0$, which is equivalent to $\|e^{t\mathbf{A}}\|_2 \le 1$.
\end{proof}

To enhance interpretability and computational efficiency, we further assume the latent observables capture quasi-orthogonal modes (e.g., POD-like modes). We parameterize $\mathbf{A}$ as a block-diagonal matrix of $2 \times 2$ blocks:
\begin{equation}
	\mathbf{A} = \bigoplus_{i=1}^{d/2} \begin{bmatrix} -\gamma_i & -\omega_i \\ \omega_i & -\gamma_i \end{bmatrix}, \quad \text{with } \gamma_i \ge 0.
	\label{eq:block_diagonal_A}
\end{equation}
Here, the learnable parameters $\gamma_i$ and $\omega_i$ explicitly represent the decay rate and frequency of the $i$-th eigenmode. The constraint $\gamma_i \ge 0$ is enforced via a non-negative activation (e.g., Softplus) during training.
It is straightforward to verify that this structure satisfies the conditions of Proposition \ref{prop:general_stability}, guaranteeing unconditional stability.

\subsection{Error Analysis and Convergence}
\label{subsec:error_analysis}

We now analyze the error propagation properties of the proposed method. Our framework can be viewed as an approximation of the true infinite-dimensional Koopman generator $\mathcal{L}$ by a finite-dimensional operator $\mathbf{A}$ on a learned subspace $\mathcal{G}_d$ spanned by the ViT encoder $\Psi_{enc}$.

The key to this approximation lies in the choice of observables. Our ViT encoder, $\Psi_{\text{enc}}: \mathcal{X} \to \mathbb{R}^d$, provides a nonlinear mapping from the high-dimensional state space $\mathcal{X}$ to a $d$-dimensional latent space of observables. Let this vector of observables be $\mathbf{g}(\mathbf{x}) = \Psi_{\text{enc}}(\mathbf{x})$. These observables span a finite-dimensional subspace, $\mathcal{G}_d = \text{span}\{g_1, \dots, g_d\}$, within the full space of all possible observables.

The action of the true generator $\mathcal{L}$ on our vector of observables $\mathbf{g}$ can be decomposed into a component that lies within our chosen subspace $\mathcal{G}_d$ and a residual component $\mathbf{r}(\mathbf{x})$ that is orthogonal to it $\mathcal{L}\mathbf{g}(\mathbf{x}) = \mathbf{A}^* \mathbf{g}(\mathbf{x}) + \mathbf{r}(\mathbf{x})$.
Here, $\mathbf{A}^*$ is the optimal finite-dimensional generator that represents the action of $\mathcal{L}$ projected onto the subspace $\mathcal{G}_d$. This is, in essence, a Galerkin projection of the infinite-dimensional operator $\mathcal{L}$ onto the subspace spanned by the learned basis functions $\{\mathbf{g}_i\}$. The term $\mathbf{r}(\mathbf{x})$ represents the intrinsic projection error (or truncation error), which is the part of the dynamics that cannot be captured by any linear model within the chosen subspace $\mathcal{G}_d$.

Let $\mathbf{g}_k = \Psi_{enc}(\mathbf{x}(t_k))$ be the true latent state at step $k$. The true evolution satisfies $\mathbf{g}_{k+1} = e^{\Delta t \mathbf{A}^*} \mathbf{g}_k + \Delta t \cdot \mathbf{r}_k$. Our learned model performs the update:
\begin{equation}
	\hat{\mathbf{g}}_{k+1} = e^{\Delta t \mathbf{A}} \hat{\mathbf{g}}_k.
\end{equation}

\begin{theorem}[Global Error Bound]
	\label{thm:global_error}
	Let $\epsilon_{reg} = \|\mathbf{A} - \mathbf{A}^*\|$ be the regression error and $\epsilon_{proj} = \sup_k \|\mathbf{r}_k\|$ be the projection error bound. Assuming the generator $\mathbf{A}$ is constructed as in Section \ref{subsec:general_stable_structure} (stable), the global prediction error at time step $n$ satisfies:
	\begin{equation}
		\|\mathbf{g}_n - \hat{\mathbf{g}}_n\| \le n \Delta t \left( \|\mathbf{g}_0\| \epsilon_{reg} + \epsilon_{proj} \right) + \mathcal{O}((\Delta t)^2).
	\end{equation}
\end{theorem}

\begin{proof}
	Let $\mathbf{e}_k = \mathbf{g}_k - \hat{\mathbf{g}}_k$ be the error at step $k$. The error evolution is:
	\begin{align}
		\mathbf{e}_{k+1} &= \mathbf{g}_{k+1} - \hat{\mathbf{g}}_{k+1} = (e^{\Delta t \mathbf{A}^*} \mathbf{g}_k + \Delta t \mathbf{r}_k) - e^{\Delta t \mathbf{A}} \hat{\mathbf{g}}_k \nonumber \\
		&= e^{\Delta t \mathbf{A}} (\mathbf{g}_k - \hat{\mathbf{g}}_k) + (e^{\Delta t \mathbf{A}^*} - e^{\Delta t \mathbf{A}}) \mathbf{g}_k + \Delta t \mathbf{r}_k \nonumber \\
		&= e^{\Delta t \mathbf{A}} \mathbf{e}_k + \Delta t (\mathbf{A}^* - \mathbf{A}) \mathbf{g}_k + \Delta t \mathbf{r}_k + \mathcal{O}((\Delta t)^2).
	\end{align}
	Taking norms and using the triangle inequality:
	\begin{equation}
		\|\mathbf{e}_{k+1}\| \le \|e^{\Delta t \mathbf{A}}\| \|\mathbf{e}_k\| + \Delta t \|\mathbf{A}^* - \mathbf{A}\| \|\mathbf{g}_k\| + \Delta t \|\mathbf{r}_k\|.
	\end{equation}
	Crucially, due to the stability constraint (Prop. \ref{prop:general_stability}), the operator is contractive: $\|e^{\Delta t \mathbf{A}}\| \le 1$. Furthermore, the state norm is bounded $\|\mathbf{g}_k\| \le \|\mathbf{g}_0\|$ (assuming a dissipative system). The recurrence relation simplifies to:
	\begin{equation}
		\|\mathbf{e}_{k+1}\| \le \|\mathbf{e}_k\| + \Delta t (\epsilon_{reg} \|\mathbf{g}_0\| + \epsilon_{proj}).
	\end{equation}
	Summing this arithmetic progression from $k=0$ to $n-1$ with $\mathbf{e}_0 = 0$ yields the result.
\end{proof}

\begin{remark}
	\label{remark-2}
	Theorem \ref{thm:global_error} theoretically grounds the stability by design of the ViT-K framework. Unlike unconstrained data-driven models where $\|e^{\Delta t \mathbf{A}}\| > 1$ inevitably induces exponential divergence $\mathcal{O}((1+\delta)^n)$, our stability constraint ensures $\|e^{\Delta t \mathbf{A}}\| \le 1$, strictly bounding error accumulation to grow at most linearly with time $T = n\Delta t$. Furthermore, the training objective rigorously targets the remaining error bounds: the projection error $\epsilon_{proj}$ is minimized by exploiting the universal approximation capability of Transformers~\cite{yun2019transformers} to learn a homeomorphism onto a maximally linear subspace, while the regression error $\epsilon_{reg}$ is minimized via the linearity loss (Eq. \ref{eq:linearity_loss-1}) through convex optimization of $\mathbf{A}$ in the latent space.
\end{remark}

\subsubsection{Minimization of Projection Error ($\epsilon_{proj}$)}
The projection error $\epsilon_{proj}(\mathbf{g}) = \|\mathcal{L}\mathbf{g} - \mathbf{A}^*\mathbf{g}\|$ quantifies the deviation of the learned observables $\mathbf{g} = \Psi_{\text{enc}}(\mathbf{x})$ from a true Koopman-invariant subspace, vanishing strictly when perfect closure is achieved.

\begin{lemma}[Approximation via ViT]
\label{lemma:approximation-vit}
	Let $\mathcal{F}_{ViT}^N$ denote the hypothesis space of Vision Transformers with $N$ parameters. Since the flow map $\Phi_t$ and the observables are continuous on a compact domain $\Omega$, the Universal Approximation Theorem for sequence-to-sequence architectures \cite{yun2019transformers} implies that for any $\delta > 0$, there exists a sufficiently large configuration $N$ and a set of parameters $\theta^*$ such that the encoder $\Psi_{\text{enc}}$ can approximate the leading Koopman eigenfunctions $\boldsymbol{\varphi} = [\varphi_1, \dots, \varphi_d]^\top$ with precision $\delta$:
	\begin{equation}
		\inf_{\Psi \in \mathcal{F}_{ViT}^N} \mathbb{E}_{\mathbf{x}} \|\Psi(\mathbf{x}) - \mathbf{T}\boldsymbol{\varphi}(\mathbf{x})\|^2 \le \delta,
	\end{equation}
	where $\mathbf{T}$ is an invertible linear transformation.
\end{lemma}

\noindent\textbf{Derivation:}
The training loss $\mathcal{L}_{total}$ includes the reconstruction term and the linearity term. The linearity term acts as a regularizer that penalizes non-closure.
\begin{equation}
	\min_{\Psi_{\text{enc}}} \sum_{k} \|\Psi_{\text{enc}}(\mathbf{x}_{k+1}) - e^{\Delta t \mathbf{A}}\Psi_{\text{enc}}(\mathbf{x}_k)\|^2.
\end{equation}
By optimizing $\Psi_{\text{enc}}$, the network seeks a manifold where the residual $\mathbf{r}_k$ is minimized. As the network capacity $N \to \infty$ and training converges, the encoder discovers a coordinate system $\mathbf{g}$ that maximally aligns with the span of the true eigenfunctions, thereby ensuring $\lim_{N \to \infty} \epsilon_{proj} = 0$.
This justifies the statement that ViT leverages its representation power to linearize the dynamics.

\subsubsection{Minimization of Regression Error ($\epsilon_{reg}$)}
The regression error $\epsilon_{reg} = \|\mathbf{A}^* - \mathbf{A}\|$ represents the deviation of our learned matrix $\mathbf{A}$ from the optimal projection $\mathbf{A}^*$ defined on the learned subspace.

\begin{lemma}[Estimation Consistency]
	For a fixed encoder $\Psi_{\text{enc}}$ (and thus a fixed feature space), the optimization of the generator $\mathbf{A}$ is a convex Least Squares problem (or Ridge Regression if weight decay is applied).
\end{lemma}

\noindent\textbf{Derivation:}
Consider the linearity loss $\mathcal{L}_{linearity}$ over a dataset of $M$ snapshots. This is equivalent to minimizing the empirical risk $
	\hat{\mathbf{A}} = \arg\min_{\mathbf{A} \in \mathcal{S}} \frac{1}{M} \sum_{k=1}^M \|\mathbf{g}_{k+1} - e^{\Delta t \mathbf{A}}\mathbf{g}_k\|^2
$,
where $\mathcal{S}$ is the set of matrices satisfying the stable structure constraints (Section \ref{subsec:general_stable_structure}).
Let $\mathbf{G} = [\mathbf{g}_1, \dots, \mathbf{g}_{M-1}]$ and $\mathbf{G}' = [\mathbf{g}_2, \dots, \mathbf{g}_M]$ be the data matrices. In the small $\Delta t$ limit, finding $\mathbf{A}$ approximates solving the normal equations $
	\mathbf{A}_{opt} \approx \frac{1}{\Delta t} (\mathbf{G}'\mathbf{G}^\dagger - \mathbf{I})$.
	
Standard statistical learning theory guarantees \cite{korda2018linear} that as the number of samples $M \to \infty$, the empirical estimator $\hat{\mathbf{A}}$ converges to the population optimal $\mathbf{A}^*$ in probability:
\begin{equation}
	\lim_{M \to \infty} P(\|\hat{\mathbf{A}} - \mathbf{A}^*\| > \eta) = 0, \quad \forall \eta > 0.
\end{equation}
Thus, $\epsilon_{reg}$ is minimized by the abundant training data and the explicit formulation of the linearity loss.

\subsubsection{Total Convergence}
\label{subsec:convergence_derivation}

In this section, we provide a rigorous derivation of the convergence rates for the ViT-K framework. Recall from Theorem \ref{thm:global_error} that the global prediction error is bounded by two distinct terms: \textbf{the projection error $\epsilon_{proj}$} (related to the model architecture) and \textbf{the regression error $\epsilon_{reg}$} (related to data sampling). We analyze the asymptotic behavior of these terms with respect to the network width/capacity $N$ and the dataset size $M$.

Let $\mathcal{L}$ be the true Koopman generator and $\hat{\mathbf{A}}$ be the empirically learned generator via the ViT-K framework. The total error can be decomposed into bias (approximation) and variance (estimation) components:
\begin{equation}
	\text{Total Error} \le \underbrace{\epsilon_{proj}(N)}_{\text{Approximation Error}} + \underbrace{\epsilon_{reg}(M)}_{\text{Estimation Error}}.
\end{equation}

Next, we will analyze the above two errors $\epsilon_{proj}$ and $\epsilon_{reg}$ as follows.

\textbf{1. Projection Error Bound: $\mathcal{O}(N^{-1})$}

The projection error $\epsilon_{proj}$ stems from the inability of a finite-sized neural network to perfectly represent the true Koopman invariant subspace.

\textbf{Assumption 1 (Spectral Regularity):} Assume the true continuous-time dynamics possess a discrete spectrum, and the leading $d$ Koopman eigenfunctions $\boldsymbol{\varphi} = [\varphi_1, \dots, \varphi_d]^\top$ belong to a spectral space $\mathcal{C}^s(\Omega)$ (smooth functions) with finite Barron norm $\|\boldsymbol{\varphi}\|_{\mathcal{B}} < \infty$.

From Lemma \ref{lemma:approximation-vit} of approximation by Vision Transformer, let $\mathcal{F}_N$ denote the function space parameterized by a Vision Transformer with width (hidden dimension) $N$. According to the universal approximation theorem for Transformers \cite{yun2019transformers} and spectral approximation bounds for neural networks \cite{barron2002universal, siegel2023optimal}, the rate at which a neural network of size $N$ approximates a smooth function $\boldsymbol{\varphi}$ in the $L^2$ norm is governed by:
\begin{equation}
	\inf_{\Psi \in \mathcal{F}_N} \|\Psi - \boldsymbol{\varphi}\|_{L^2(\Omega)} \le \frac{C_{\boldsymbol{\varphi}}}{N},
\end{equation}
where $C_{\boldsymbol{\varphi}}$ depends on the smoothness (Barron norm) of the eigenfunctions but is independent of $N$.

Since the projection residual $\mathbf{r}_k$ in $\mathcal{L}\mathbf{g}(\mathbf{x}) = \mathbf{A}^* \mathbf{g}(\mathbf{x}) + \mathbf{r}(\mathbf{x})$ vanishes if the observable $\mathbf{g}$ lies exactly in the span of eigenfunctions, the projection error is directly proportional to the approximation error of the eigenfunctions:

\begin{eqnarray}
	\epsilon_{proj}(N) \triangleq \mathbb{E}_{\mathbf{x}} \|\mathbf{r}(\mathbf{x})\| \le C_1 \inf_{\Psi \in \mathcal{F}_N} \|\Psi - \boldsymbol{\varphi}\| \le \mathcal{O}(N^{-1}).
\end{eqnarray}

This establishes that increasing the model capacity $N$ linearly reduces the projection error (bias).

\textbf{2. Regression Error Bound: $\mathcal{O}(M^{-1/2})$}

The regression error $\epsilon_{reg}$ arises from estimating the operator $\mathbf{A}$ using a finite dataset of $M$ snapshot pairs $\{(\mathbf{x}_k, \mathbf{x}_{k+1})\}_{k=1}^M$. This is a statistical estimation problem. For a fixed encoder $\Psi$ (and thus fixed features $\mathbf{g}$), minimizing the linearity loss is equivalent to solving a Least Squares problem for $\mathbf{A}$:
\begin{equation}
	\hat{\mathbf{A}} = \arg\min_{\mathbf{A}} \frac{1}{M} \sum_{k=1}^M \|\mathbf{g}_{k+1} - e^{\Delta t \mathbf{A}}\mathbf{g}_k\|^2.
\end{equation}
Let $\mathbf{A}^*$ be the population minimizer (infinite data limit). The difference $\|\hat{\mathbf{A}} - \mathbf{A}^*\|$ is the estimation error.

Assuming the data samples are i.i.d. drawn from the underlying distribution $\mu$, and the feature vectors are bounded $\|\mathbf{g}\| \le B$. By standard results in statistical learning theory (Rademacher complexity bounds for linear hypothesis classes \cite{bartlett2002rademacher}), the generalization gap between the empirical risk and true risk scales with the square root of the sample size.
Specifically, for matrix regression problems, the bound is given by $
	\mathbb{E}[\|\hat{\mathbf{A}} - \mathbf{A}^*\|_F] \le \frac{C_{noise} \cdot B}{\sqrt{M}}$,
where $C_{noise}$ represents the noise variance or the intrinsic residual of the system, and $\|\cdot\|_F$ is the Frobenius norm.

Thus, we obtain the estimation error bound $\epsilon_{reg}(M) \le \mathcal{O}(M^{-1/2})$.
This establishes that the error in learning the dynamics matrix decreases with the square root of the data quantity.

\textbf{3. Total Error Convergence}

Substituting the derived bounds back into the global error inequality from Theorem \ref{thm:global_error}, we obtain the final convergence rate:

\begin{equation}
	\text{Total Error} \le T \cdot \left( \|\mathbf{g}_0\| \epsilon_{reg}(M) + \epsilon_{proj}(N) \right) 
	\le T \cdot \left( C_2 M^{-1/2} + C_1 N^{-1} \right).
\end{equation}

The total error bound from Theorem \ref{thm:global_error} converges as:
\begin{equation}
	\text{Total Error} \propto \underbrace{\mathcal{O}(N^{-1})}_{\text{ViT Capacity } (\epsilon_{proj})} + \underbrace{\mathcal{O}(M^{-1/2})}_{\text{Data Sampling } (\epsilon_{reg})}.
\end{equation}
This derivation highlights the dual requirement for high-fidelity modeling: (1) A sufficiently large ViT model ($N$) is needed to discover the correct Koopman invariant subspace, and (2) Sufficient training data ($M$) is required to accurately identify the linear dynamics within that subspace.

\section{Numerical experiments}
\label{sec:numerical}

In this section, we present a series of numerical experiments to validate the accuracy and long-term stability of the proposed ViT-K framework. The performance of ViT-K method is quantitatively assessed using three standard metrics. For each quantity $f \in \{u_1, u_2, p, \phi\}$ at a given time $t_n$, we define:
\begin{align*}
	\mathrm{MSE} &= \frac{1}{N}\sum_{i,j}(\hat{f}_{i,j}-f_{i,j})^2,
	& \mathrm{MAE} &= \frac{1}{N}\sum_{i,j}|\hat{f}_{i,j}-f_{i,j}|, \\
	\mathrm{Rel}\,L^2 &= \frac{\|\hat{f}-f\|_2}{\|f\|_2},
	& \mathrm{RMSE}(t) &= \sqrt{ \frac{1}{N} \sum_{i=1}^{N} \left( \hat{u}_i(t) - u_i(t) \right)^2 }.
\end{align*}
where $N$ denotes the total number of spatial collocation points.

\subsection{Example 1: Time-Dependent Stokes--Darcy System}
\label{subsec:ex1_stokes_darcy}

We first benchmark the ViT-K framework on a coupled time-dependent Stokes-Darcy problem defined on the domain $\Omega = [0, \pi] \times [-1, 1]$. The domain is partitioned into the porous medium region $\Omega_D = [0, \pi] \times [0, 1]$ and the free-flow region $\Omega_S = [0, \pi] \times [-1, 0]$. The physical parameters are set to unity: $\nu=1, g=1, k=1$, and the interface height $z=0$.

We construct an analytical solution to validate the model's ability to capture dissipative dynamics. The exact fields for the hydraulic head $\phi_D$, fluid velocity $\vec{u}_S = [u_1, u_2]^T$, and pressure $p_S$ are given by:
\begin{align}
	\phi_D &= (e^y - e^{-y}) \sin(x) e^{-t}, \qquad u_1 = \frac{k}{\pi} \sin(2\pi y) \cos(x) e^{-t}, \nonumber\\
	u_2 &= \left(-2k + \frac{k}{\pi^2} \sin^2(\pi y)\right) \sin(x) e^{-t}, \qquad p = 0.\nonumber
\end{align}
The dataset consists of snapshots generated over the interval $t \in [0, 2.0]$. The initial phase $t \in [0, 1.0]$ serves as the “training horizon”, while the subsequent phase $t \in (1.0, 2.0]$ serves as the “extrapolation horizon” to test long-term stability. So our method is Few-Shot Learning framework with only 10 snapshots for training. The spatial domain is discretized using a uniform $96 \times 96$ grid. All physical fields are normalized by their maximum amplitudes to ensure numerical stability and scale invariance. Specifically, for a physical field \( u(\mathbf{x}, t) \), the normalized field \( \tilde{u}(\mathbf{x}, t) \) is defined as
\[
\tilde{u}(\mathbf{x}, t)
\;=\;
\frac{u(\mathbf{x}, t)}
{\displaystyle \max_{\mathbf{x},\, t \in [0,\,1.0]} \left| u(\mathbf{x}, t) \right|}.
\]
A binary mask indicating the subdomains is appended as an auxiliary input channel.

\textbf{Implementation Details:}
The ViT-K model is instantiated with a Vision Transformer encoder featuring a patch size of $16 \times 16$, an embedding dimension of $192$, and depth of $6$ layers with $6$ attention heads. The temporal evolution in the latent space is governed by the continuous-time Koopman operator $\mathcal{K}(\Delta t) = \exp(\Delta t \mathbf{A})$. 
Training is performed for $1200$ epochs using the AdamW optimizer (learning rate $5 \times 10^{-4}$, weight decay $2 \times 10^{-5}$) and a cosine annealing schedule with a 10-epoch warmup. The loss function weights are empirically set to $w_{u_1}=3.0, w_{u_2}=5.0, w_{p}=0.05$, and $w_{\phi}=1.0$ to balance the gradients across different magnitudes.

\textbf{Prediction Accuracy and Stability:} 
Tables \ref{tab:sd_dt005} and \ref{tab:sd_dt020} summarize the quantitative evaluation of the model trained with a time step $\Delta t \in \{0.05, 0.2\}$. For our few-shot learning framework, only 20 (with $\Delta t=0.05$) or 5 (with $\Delta t=0.2$) snapshots are used in training. Figure~\ref{fig:sd_nsd_loss} shows the convergence history, demonstrating a stable and efficient optimization process. From Table~\ref{tab:sd_dt020}, we find that the ViT-K framework exhibits exceptional reconstruction fidelity in the interpolation regime ($t \le 1.0$). The relative errors for velocity and hydraulic head are consistently below $0.8\%$, with the Mean Squared Error (MSE) remaining on the order of $\mathcal{O}(10^{-6})$. 

In the extrapolation regime ($t > 1.0$), we observe a natural error growth due to the recursive nature of the prediction. However, the error accumulation remains bounded. At $t=2.0$, the maximum relative error is constrained within $10\%\sim 14\%$, and the absolute MSE remains low ($\mathcal{O}(10^{-5})$). This confirms that the learned Koopman operator successfully captures the global stability of the dissipative system without exponential divergence.
	
Note that the relative error for pressure $p$ is undefined (marked as $-$) because the exact solution is trivial ($p_S=0$). The absolute MSE for pressure remains negligible ($\approx 5 \times 10^{-5}$), indicating the model correctly maintains the zero-pressure condition.

Furthermore, Figure~\ref{fig:error-linear-1} illustrates the temporal evolution of the Root Mean Square Error (RMSE) during the extrapolation phase. Notably, the error exhibits an almost linear growth trend over time, as illustrated by the Theorem \ref{thm:global_error} and Remark \ref{remark-2}. 

\begin{table}[htbp] 
	\footnotesize
	\centering 
	\caption{Quantitative assessment of prediction quality ($\Delta t = 0.05$).}
	\label{tab:sd_dt005}
	\setlength{\tabcolsep}{5pt} 
	\begin{tabular}{lccccc}
		\toprule 
		Channel & Time & MSE & MAE & Max Error & Relative Error (\%) \\
		\midrule
		\multirow{2}{*}{$u_1$} 
		& $t=1.0$ & $2.32 \times 10^{-6}$ & $1.13 \times 10^{-3}$ & $6.44 \times 10^{-3}$ & $0.81$ \\
		& $t=2.0$ & $6.33 \times 10^{-5}$ & $5.86 \times 10^{-3}$ & $4.22 \times 10^{-2}$ & $11.75$ \\
		\midrule
		\multirow{2}{*}{$u_2$} 
		& $t=1.0$ & $1.27 \times 10^{-6}$ & $8.81 \times 10^{-4}$ & $4.22 \times 10^{-3}$ & $0.43$ \\
		& $t=2.0$ & $8.86 \times 10^{-5}$ & $6.61 \times 10^{-3}$ & $4.99 \times 10^{-2}$ & $10.14$ \\
		\midrule
		\multirow{2}{*}{$p$} 
		& $t=1.0$ & $5.45 \times 10^{-6}$ & $1.72 \times 10^{-3}$ & $1.14 \times 10^{-2}$ & -- \\
		& $t=2.0$ & $5.66 \times 10^{-5}$ & $5.72 \times 10^{-3}$ & $3.69 \times 10^{-2}$ & -- \\
		\midrule
		\multirow{2}{*}{$\phi$} 
		& $t=1.0$ & $1.54 \times 10^{-6}$ & $9.69 \times 10^{-4}$ & $4.90 \times 10^{-3}$ & $0.85$ \\
		& $t=2.0$ & $6.27 \times 10^{-5}$ & $5.71 \times 10^{-3}$ & $4.07 \times 10^{-2}$ & $15.14$ \\
		\bottomrule
	\end{tabular}
	\vspace{-2em}
\end{table}

\begin{table}[htbp]
	\footnotesize
	\centering
	\caption{Quantitative assessment of prediction quality ($\Delta t = 0.2$).}
	\label{tab:sd_dt020}
	\setlength{\tabcolsep}{5pt}
	\begin{tabular}{lccccc}
		\toprule
		Channel & Time & MSE & MAE & Max Error & Relative Error (\%) \\
		\midrule
		\multirow{2}{*}{$u_1$}
		& $t=1.0$ & $2.29 \times 10^{-6}$ & $1.12 \times 10^{-3}$ & $7.36 \times 10^{-3}$ & $0.74$ \\
		& $t=2.0$ & $4.16 \times 10^{-5}$ & $4.75 \times 10^{-3}$ & $3.22 \times 10^{-2}$ & $9.53$ \\
		\midrule
		\multirow{2}{*}{$u_2$}
		& $t=1.0$ & $1.15 \times 10^{-6}$ & $0.83 \times 10^{-3}$ & $4.15 \times 10^{-3}$ & $0.38$ \\
		& $t=2.0$ & $5.04 \times 10^{-5}$ & $4.93 \times 10^{-3}$ & $3.91 \times 10^{-2}$ & $7.65$ \\
		\midrule
		\multirow{2}{*}{$p$}
		& $t=1.0$ & $4.30 \times 10^{-6}$ & $1.53 \times 10^{-3}$ & $1.01 \times 10^{-2}$ & -- \\
		& $t=2.0$ & $4.28 \times 10^{-5}$ & $4.98 \times 10^{-3}$ & $3.22 \times 10^{-2}$ & -- \\
		\midrule
		\multirow{2}{*}{$\phi$}
		& $t=1.0$ & $1.41 \times 10^{-6}$ & $0.93 \times 10^{-3}$ & $4.91 \times 10^{-3}$ & $0.75$ \\
		& $t=2.0$ & $4.12 \times 10^{-5}$ & $4.67 \times 10^{-3}$ & $3.46 \times 10^{-2}$ & $12.27$ \\
		\bottomrule
	\end{tabular}
	\vspace{-1em}
\end{table}

\begin{figure}[htbp]
	\centering
	\includegraphics[width=0.8\linewidth]{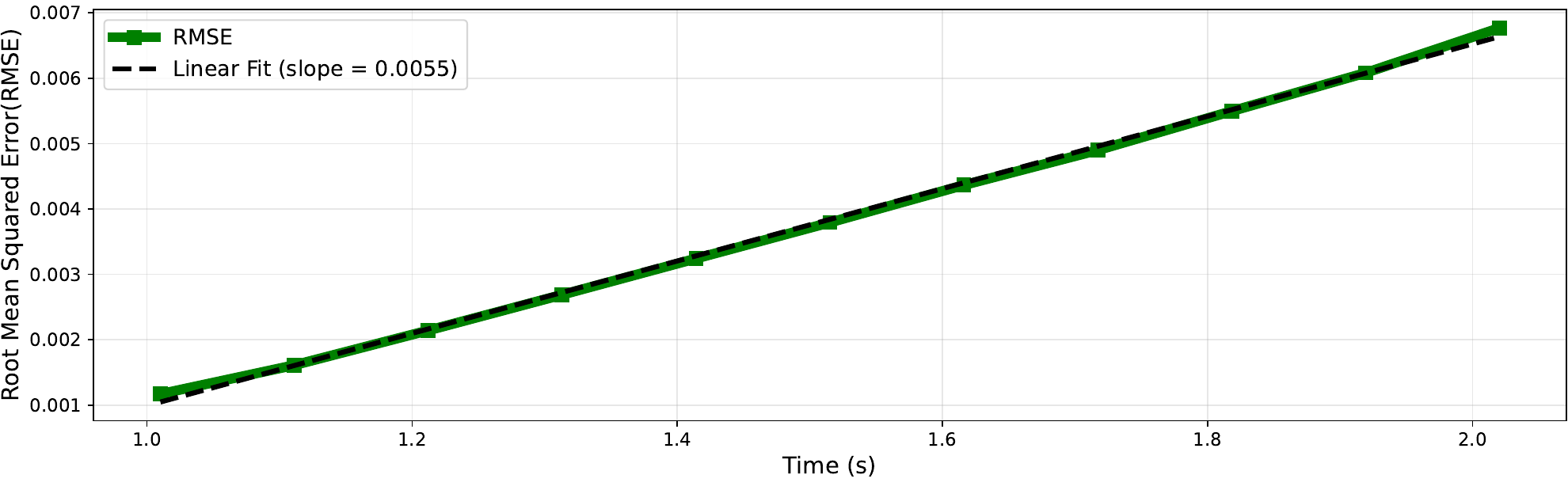}
	\caption{Root Mean Square Error (RMSE) vs. Extrapolation Time}
	\label{fig:error-linear-1}
	\vspace{-2em}
\end{figure}

\textbf{Impact of Temporal Sampling Step $\Delta t$:}
We further investigate the sensitivity of the model to the sampling step $\Delta t$. A comparative summary is provided in Table \ref{tab:dt_two_times}. An interesting phenomenon is observed in the long-term extrapolation ($t=2.0$): the model with the largest time step ($\Delta t = 0.2$) achieves the lowest prediction error. For instance, the relative error for $u_1$ at $t=2.0$ decreases from $11.75\%$ ($\Delta t=0.05$) to $9.53\%$ ($\Delta t=0.2$).
This counter-intuitive result can be attributed to the trade-off between discretization error and error accumulation. While a smaller $\Delta t$ offers finer temporal resolution, it requires more recursive steps ($N_{steps} = T/\Delta t$) to reach the final time $T$. Since the operator prediction $\hat{\mathbf{g}}_{k+1} = \mathbf{K}\mathbf{g}_k$ introduces a small error $\epsilon$ at each step, the total accumulated error after $N$ steps scales with $N$. A larger $\Delta t$ reduces the number of recurrent operations, thereby mitigating the accumulation of floating-point and approximation errors over long horizons.
\begin{table}[htbp]
	\footnotesize
	\centering
	\caption{Effect of temporal sampling step $\Delta t$ on prediction accuracy. Larger $\Delta t$ yields better long-term stability due to reduced error accumulation steps.}
	\label{tab:dt_two_times}
    \setlength{\tabcolsep}{3pt} 
	\begin{tabular}{lcccccc}
		\toprule
		& \multicolumn{6}{c}{Temporal sampling step $\Delta t$} \\
		\cmidrule(lr){2-7}
		Channel
		& \multicolumn{3}{c}{$t = 1.0$ (Interpolation)}
		& \multicolumn{3}{c}{$t = 2.0$ (Extrapolation)} \\
		\cmidrule(lr){2-4} \cmidrule(lr){5-7}
		& $0.05$ & $0.1$ & $0.2$
		& $0.05$ & $0.1$ & $0.2$ \\
		\midrule
		$u_1$ (MSE)
		& $2.32{\times}10^{-6}$ & \textbf{$2.22{\times}10^{-6}$} & $2.29{\times}10^{-6}$
		& $6.33{\times}10^{-5}$ & $5.27{\times}10^{-5}$ & \textbf{$4.16{\times}10^{-5}$} \\
		$u_1$ (Rel.\%)
		& $0.81$ & $0.77$ & \textbf{$0.74$}
		& $11.75$ & $10.72$ & \textbf{$9.53$} \\
		\midrule
		$u_2$ (MSE)
		& $1.27{\times}10^{-6}$ & \textbf{$9.30{\times}10^{-7}$} & $1.15{\times}10^{-6}$
		& $8.86{\times}10^{-5}$ & $7.39{\times}10^{-5}$ & \textbf{$5.04{\times}10^{-5}$} \\
		$u_2$ (Rel.\%)
		& $0.43$ & \textbf{$0.36$} & $0.38$
		& $10.14$ & $9.26$ & \textbf{$7.65$} \\
		\bottomrule
	\end{tabular}\vspace{-2em}
\end{table}
\begin{figure}[t]
	\centering
	\includegraphics[width=0.48\linewidth]{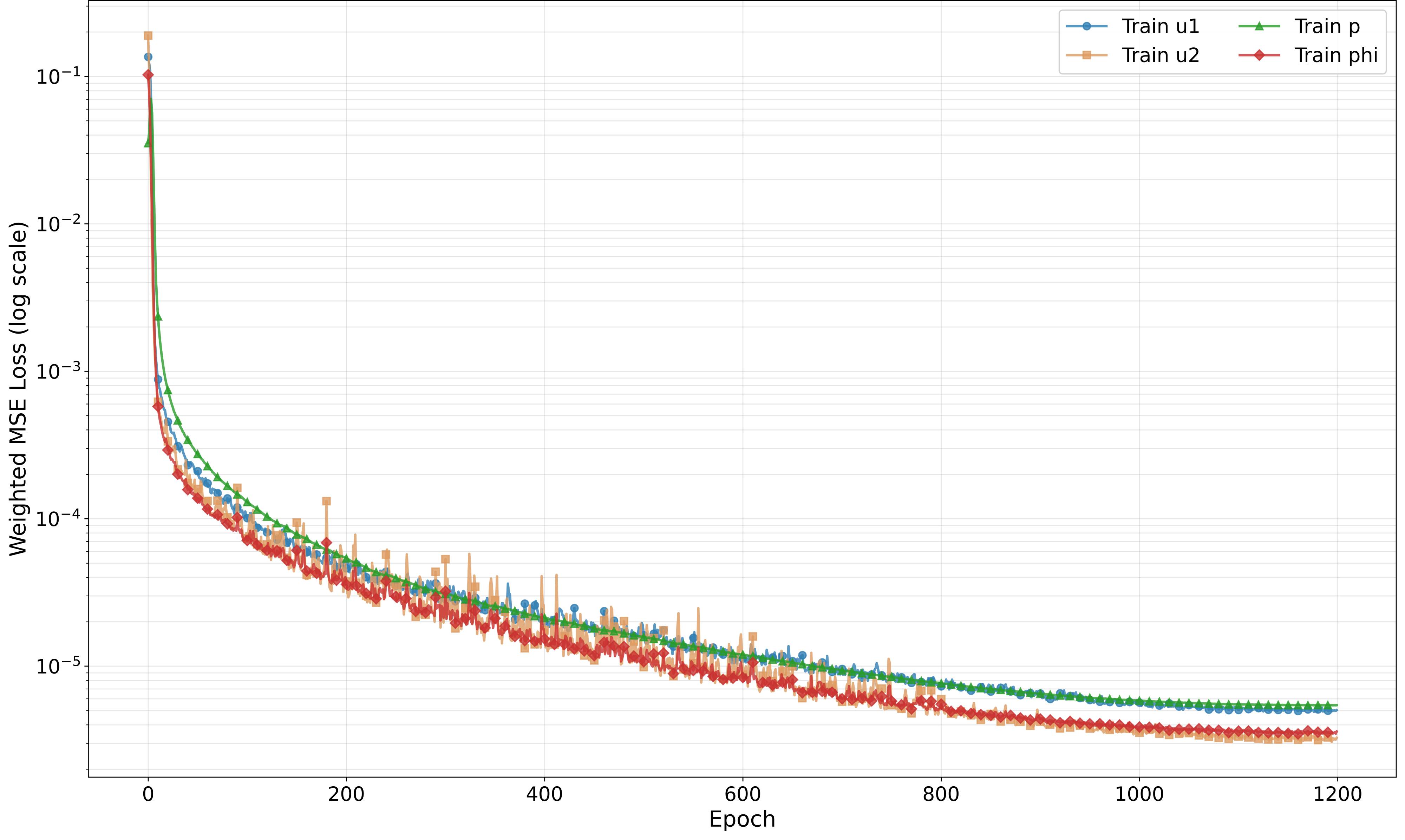}
	\includegraphics[width=0.48\linewidth]{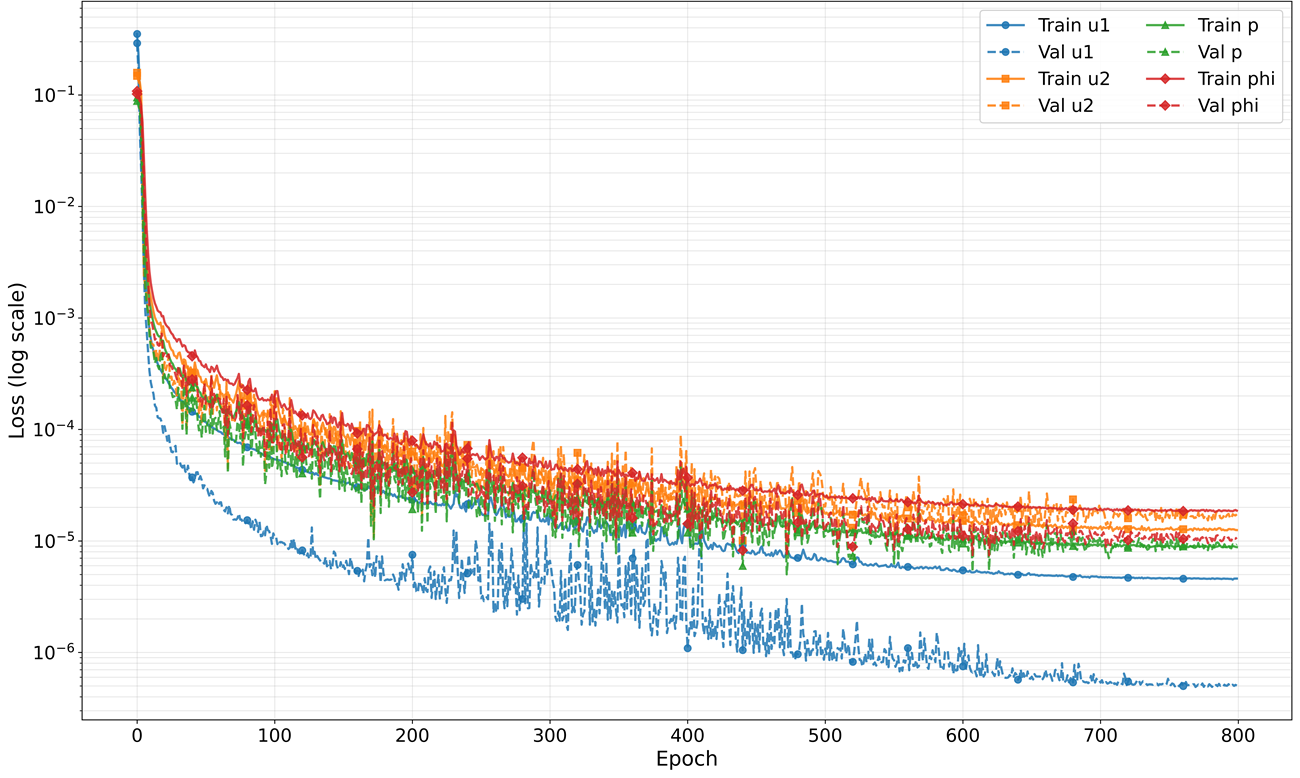}
	\caption{Convergence history of the component-wise training losses, showing stable optimization. Left is for Stokes--Darcy System and right is for Navier--Stokes--Darcy System.}
	\label{fig:sd_nsd_loss}
	\vspace{-3em}
\end{figure}

\subsection{Example 2: Time-Dependent Navier--Stokes--Darcy System}
\label{subsec:ex2_nsd}
We extend our validation to the time-dependent Navier--Stokes--Darcy (NSD) system, which introduces nonlinear convective transport in the free-flow region. The computational domain is defined as $\Omega = [0,1]\times[-0.25,0.75]$, partitioned into the Darcy region $\Omega_D = [0,1]\times[0,0.75]$ and the Stokes region $\Omega_S = [0,1]\times[-0.25,0]$. The model parameters are set to unity: $\alpha=1, \nu=1, g=1, z=0$, and $\mathbb{K}=k\mathbf{I}$ with $k=1$.

Unlike the dissipative dynamics in Example 1, here we construct a time-periodic manufactured solution to test the model's ability to capture persistent oscillatory dynamics (limit cycles). The exact solutions satisfying the interface conditions (including the Beavers-Joseph law) are:
\begin{align}
	\phi_D &= [\,2 - \pi\sin(\pi x)\,]\,[\,-y + \cos(\pi(1-y))\,]\cos(2\pi t), \\
	\vec{u}_S &= 
	\begin{bmatrix}
		x^2y^2 + e^{-y} \\
		-\tfrac{2}{3}xy^3 + 2 - \pi\sin(\pi x)
	\end{bmatrix}
	\cos(2\pi t), \\
	p &= -[\,2 - \pi\sin(\pi x)\,]\cos(2\pi y)\cos(2\pi t).
\end{align}

\textbf{Implementation Details:}
The discretization and network architecture (ViT depth, patch size) remain consistent with Example 1 to ensure comparability. A critical distinction in this experiment lies in the parameterization of the Koopman generator $\mathbf{A}$. Since the underlying dynamics are periodic, the system's energy should be conserved (or oscillate) in the asymptotic limit rather than decay. Therefore, we enforce the generator $\mathbf{A}$ to be dominated by skew-symmetric blocks (representing rotations). The diagonal decay rates are initialized close to zero, encouraging the learning of purely imaginary eigenvalues $\lambda \approx \pm i\omega$. This enforces an energy-conserving harmonic evolution in the latent space.
The loss weights are adjusted to $w_{u_1}=3.0, w_{u_2}=3.0, w_{p}=1.0$, and $w_{\phi}=1.0$. Figure \ref{fig:sd_nsd_loss} shows the convergence history, demonstrating a stable and efficient optimization process.

\textbf{Prediction Accuracy and Stability:} 
Tables \ref{tab:nsd}, and \ref{tab:nsd_dt02} detail the quantitative performance for time steps $\Delta t \in \{0.1, 0.2\}$. A comparative summary is provided in Table \ref{tab:nsd_dt_two_times}. We can find that the ViT-K framework achieves high accuracy and remarkable stability. For the time step $\Delta t = 0.1$, the relative $L^2$ error for velocity components ($u_1, u_2$) remains below $0.9\%$ throughout both the training ($t \le 1.0$) and extrapolation ($t > 1.0$) phases. Notably, unlike the purely dissipative case, the errors do not grow monotonically but exhibit a stable, bounded behavior. This confirms that the rotation-based Koopman operator successfully locks onto the phase and frequency of the periodic driver $\cos(2\pi t)$. 
\begin{table}
	\footnotesize
	\centering
	\caption{Quantitative assessment of prediction quality ($\Delta t = 0.1$).}
	\label{tab:nsd}
	\setlength{\tabcolsep}{5pt}
	\begin{tabular}{lccccc}
		\toprule
		Channel & Time & MSE & MAE & Max Error & Relative Error (\%) \\
		\midrule
		\multirow{3}{*}{$u_1$}
		& $t=0.0$ & $6.00 \times 10^{-6}$ & $1.90 \times 10^{-3}$ & $1.02 \times 10^{-2}$ & $0.29$ \\
		& $t=1.0$ & $5.94 \times 10^{-6}$ & $1.90 \times 10^{-3}$ & $9.94 \times 10^{-3}$ & $0.29$ \\
		& $t=2.0$ & $6.01 \times 10^{-6}$ & $1.90 \times 10^{-3}$ & $1.01 \times 10^{-2}$ & $0.29$ \\
		\midrule
		\multirow{3}{*}{$u_2$}
		& $t=0.0$ & $1.87 \times 10^{-5}$ & $3.46 \times 10^{-3}$ & $1.66 \times 10^{-2}$ & $0.88$ \\
		& $t=1.0$ & $1.87 \times 10^{-5}$ & $3.46 \times 10^{-3}$ & $1.75 \times 10^{-2}$ & $0.89$ \\
		& $t=2.0$ & $1.90 \times 10^{-5}$ & $3.49 \times 10^{-3}$ & $1.72 \times 10^{-2}$ & $0.89$ \\
		\midrule
		\multirow{3}{*}{$p$}
		& $t=0.0$ & $1.35 \times 10^{-5}$ & $2.82 \times 10^{-3}$ & $1.56 \times 10^{-2}$ & $1.07$ \\
		& $t=1.0$ & $1.33 \times 10^{-5}$ & $2.80 \times 10^{-3}$ & $1.45 \times 10^{-2}$ & $1.08$ \\
		& $t=2.0$ & $1.37 \times 10^{-5}$ & $2.85 \times 10^{-3}$ & $1.48 \times 10^{-2}$ & $1.08$ \\
		\midrule
		\multirow{3}{*}{$\phi$}
		& $t=0.0$ & $2.77 \times 10^{-5}$ & $4.09 \times 10^{-3}$ & $2.98 \times 10^{-2}$ & $1.49$ \\
		& $t=1.0$ & $2.74 \times 10^{-5}$ & $4.06 \times 10^{-3}$ & $2.87 \times 10^{-2}$ & $1.50$ \\
		& $t=2.0$ & $2.82 \times 10^{-5}$ & $4.12 \times 10^{-3}$ & $2.85 \times 10^{-2}$ & $1.51$ \\
		\bottomrule
	\end{tabular}
	\vspace{-2em}
\end{table}

\begin{table}
	\footnotesize
	\centering
	\caption{Quantitative assessment of prediction quality ($\Delta t = 0.2$).}
	\label{tab:nsd_dt02}
	\setlength{\tabcolsep}{5pt}
	\begin{tabular}{lccccc}
		\toprule
		Channel & Time & MSE & MAE & Max Error & Relative Error (\%) \\
		\midrule
		\multirow{3}{*}{$u_1$}
		& $t=0.0$ & $6.04 \times 10^{-6}$ & $1.87 \times 10^{-3}$ & $1.35 \times 10^{-2}$ & $0.29$ \\
		& $t=1.0$ & $6.74 \times 10^{-5}$ & $6.22 \times 10^{-3}$ & $6.04 \times 10^{-2}$ & $1.25$ \\
		& $t=2.0$ & $6.03 \times 10^{-6}$ & $1.87 \times 10^{-3}$ & $1.34 \times 10^{-2}$ & $0.29$ \\ 
		\midrule
		\multirow{3}{*}{$u_2$}
		& $t=0.0$ & $2.99 \times 10^{-5}$ & $4.17 \times 10^{-3}$ & $2.46 \times 10^{-2}$ & $1.12$ \\
		& $t=1.0$ & $1.02 \times 10^{-4}$ & $8.05 \times 10^{-3}$ & $4.56 \times 10^{-2}$ & $2.70$ \\
		& $t=2.0$ & $2.99 \times 10^{-5}$ & $4.17 \times 10^{-3}$ & $2.47 \times 10^{-2}$ & $1.12$ \\ 
		\midrule
		\multirow{3}{*}{$p$}
		& $t=0.0$ & $1.98 \times 10^{-5}$ & $3.15 \times 10^{-3}$ & $3.00 \times 10^{-2}$ & $1.30$ \\
		& $t=1.0$ & $6.81 \times 10^{-5}$ & $6.39 \times 10^{-3}$ & $4.05 \times 10^{-2}$ & $3.14$ \\
		& $t=2.0$ & $1.98 \times 10^{-5}$ & $3.15 \times 10^{-3}$ & $2.99 \times 10^{-2}$ & $1.30$ \\ 
		\midrule
		\multirow{3}{*}{$\phi$}
		& $t=0.0$ & $3.82 \times 10^{-5}$ & $4.57 \times 10^{-3}$ & $3.62 \times 10^{-2}$ & $1.75$ \\
		& $t=1.0$ & $9.74 \times 10^{-5}$ & $7.55 \times 10^{-3}$ & $5.77 \times 10^{-2}$ & $3.65$ \\
		& $t=2.0$ & $3.83 \times 10^{-5}$ & $4.57 \times 10^{-3}$ & $3.64 \times 10^{-2}$ & $1.75$ \\
		\bottomrule
	\end{tabular}
	\vspace{-2em}
\end{table}

\textbf{Impact of Temporal Sampling Step $\Delta t$:}
Table \ref{tab:nsd_dt_two_times} reveals an important contrast to Example 1 with different $\Delta t$. Here, the smallest time step ($\Delta t = 0.05$) yields the highest accuracy (e.g., $0.23\%$ relative error for $u_1$), whereas the largest step ($\Delta t = 0.2$) shows a degradation in accuracy (rising to $1.25\%$ at $t=1.0$).
This behavior is consistent with the simulation of oscillatory systems: larger time steps introduce phase dispersion errors in the discrete matrix exponential $e^{\Delta t \mathbf{A}}$, leading to a gradual phase drift over time. Thus, for periodic dynamics, a finer temporal resolution is preferable to maintain phase fidelity.
\begin{table}
	\footnotesize
	\centering
	\caption{Effect of temporal sampling step $\Delta t$ on prediction accuracy for the Navier--Stokes--Darcy model. Results are reported at interpolation ($t=1.0$) and extrapolation ($t=2.0$) stages.}
	\label{tab:nsd_dt_two_times}
	\setlength{\tabcolsep}{2pt} 
	\small
	\begin{tabular}{lcccccc}
		\toprule
		& \multicolumn{6}{c}{Temporal sampling step $\Delta t$} \\
		\cmidrule(lr){2-7}
		Channel
		& \multicolumn{3}{c}{$t = 1.0$ (Interpolation)}
		& \multicolumn{3}{c}{$t = 2.0$ (Extrapolation)} \\
		\cmidrule(lr){2-4} \cmidrule(lr){5-7}
		& $0.05$ & $0.1$ & $0.2$
		& $0.05$ & $0.1$ & $0.2$ \\
		\midrule
		$u_1$ (MSE)
		& \textbf{$3.85\times10^{-6}$} & $5.94\times10^{-6}$ & $6.74\times10^{-5}$
		& \textbf{$3.93\times10^{-6}$} & $6.01\times10^{-6}$ & $6.03\times10^{-6}$ \\
		$u_1$ (Rel.\%)
		& \textbf{$0.23$} & $0.29$ & $1.25$
		& \textbf{$0.23$} & $0.29$ & $0.29$ \\
		\midrule
		$u_2$ (MSE)
		& \textbf{$1.28\times10^{-5}$} & $1.87\times10^{-5}$ & $1.02\times10^{-4}$
		& \textbf{$1.31\times10^{-5}$} & $1.90\times10^{-5}$ & $2.99\times10^{-5}$ \\
		$u_2$ (Rel.\%)
		& \textbf{$0.74$} & $0.89$ & $2.70$
		& \textbf{$0.74$} & $0.89$ & $1.12$ \\
		\midrule
		$p$ (MSE)
		& \textbf{$8.23\times10^{-6}$} & $1.33\times10^{-5}$ & $6.81\times10^{-5}$
		& \textbf{$8.41\times10^{-6}$} & $1.37\times10^{-5}$ & $1.98\times10^{-5}$ \\
		$p$ (Rel.\%)
		& \textbf{$0.85$} & $1.08$ & $3.14$
		& \textbf{$0.85$} & $1.08$ & $1.30$ \\
		\midrule
		$\phi$ (MSE)
		& \textbf{$2.30\times10^{-5}$} & $2.74\times10^{-5}$ & $9.74\times10^{-5}$
		& \textbf{$2.36\times10^{-5}$} & $2.82\times10^{-5}$ & $3.83\times10^{-5}$ \\
		$\phi$ (Rel.\%)
		& \textbf{$1.38$} & $1.50$ & $3.65$
		& \textbf{$1.38$} & $1.51$ & $1.75$ \\
		\bottomrule
	\end{tabular}
	\vspace{-2em}
\end{table}

\textbf{Comparison with Baselines:}
	We conducted a comparative analysis with two representative baseline models: the FNO \cite{li2020fourier} and the ConvLSTM \cite{shi2015convolutional}. The FNO baseline is configured with 12 Fourier modes and 4 spectral convolution layers (width 64), utilizing global frequency-domain operations to capture spatiotemporal dynamics. The ConvLSTM baseline employs 64 hidden channels and optimized recurrent convolutional units. All models were trained using the AdamW optimizer with a cosine annealing learning rate scheduler for 300 epochs.

Figure \ref{fig:ViT-K+FNO+LSTM} shows the time-domain evolution of the velocity component $u_1$ at a representative point $(0.5,0.08)$ within the free-flow region. Although all three models successfully fit the training interval ($0–1$ s), FNO and ConvLSTM began to deviate during the extrapolation phase ($1–3$ s). As shown in Table \ref{tab:ViT-K+FNO+LSTM}, ViT-K produces smaller errors in the prediction intervals compared to the baseline model.
%
\begin{figure}[ht] 
	\centering
	\begin{minipage}[ht]{0.45\textwidth}
		\centering
		\includegraphics[width=\textwidth]{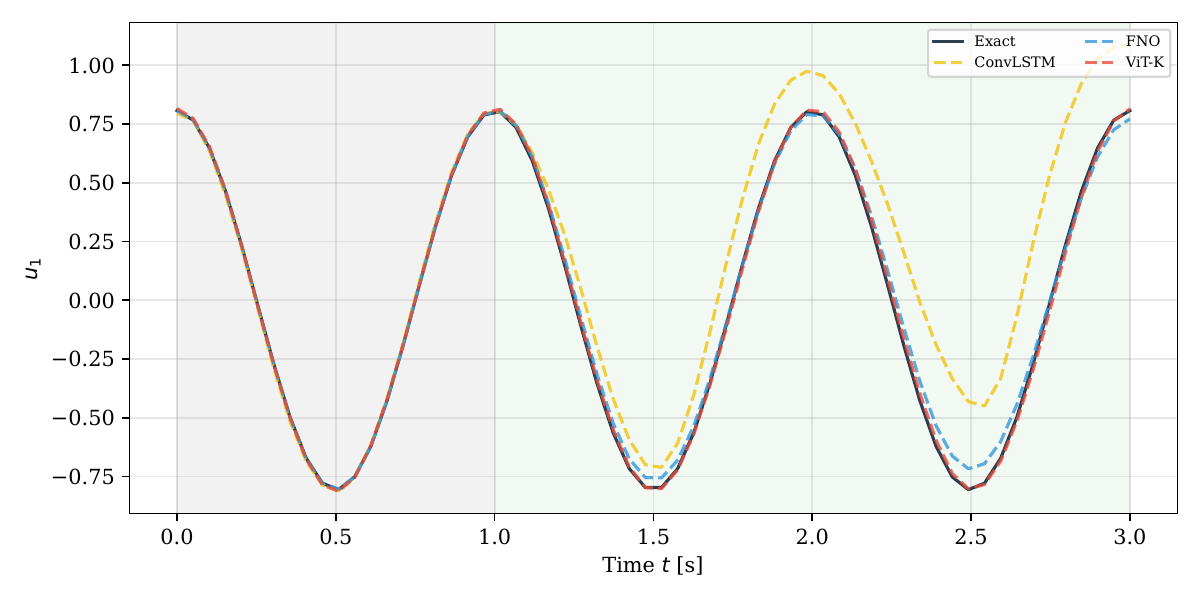}
		\caption{Temporal evolution of velocity $u_1$ at $(0.5, -0.08)$ using ViT-K, FNO, and ConvLSTM.}
		\label{fig:ViT-K+FNO+LSTM}
	\end{minipage}
	\hfill
	\begin{minipage}[ht]{0.52\textwidth}
		\centering
		\captionof{table}{Quantitative comparison of extrapolation errors ($1$--$3$ s) for ViT-K, FNO and ConvLSTM.}
		\label{tab:ViT-K+FNO+LSTM}
		\setlength{\tabcolsep}{3pt}
		\footnotesize 
		\begin{tabular}{lccc}
			\toprule
			Metric & ViT-K & FNO & ConvLSTM \\
			\midrule
			MSE & $5.06 \times 10^{-5}$ & $2.23 \times 10^{-4}$ & $5.19 \times 10^{-3}$ \\
			\bottomrule
		\end{tabular}
	\end{minipage}
	\vspace{-3em}
\end{figure}

\subsection{Example 3: Flow in Heterogeneous Karst Media}
\label{subsec:ex3_karst}

To evaluate ViT-K's adaptability to non-convex boundaries and strong free-flow/porous-medium coupling, we consider the Y-shaped karst aquifer benchmark (Figure \ref{ex3-domain}). Unlike idealized rectangular domains, this case features a high-permeability fracture conduit embedded within a low-permeability porous matrix, challenging the framework to simultaneously resolve the Beavers--Joseph (BJ) interface condition along irregular polygons and extrapolate temporal dynamics. The domain $\Omega = [0,1]^2$ is partitioned into a Navier--Stokes free-flow region ($\Omega_S$) and a Darcy porous matrix ($\Omega_D$). The highly heterogeneous interaction occurs at the sharp interface $\Gamma = \partial\Omega_S \cap \partial\Omega_D$, where the BJ condition is explicitly imposed to ensure momentum and mass consistency. The system is driven by piecewise velocity Dirichlet conditions on the fracture boundaries $\partial\Omega_S \setminus \Gamma$ with global mass balance enforced. A homogeneous hydraulic head $\phi = 0$ is maintained on the exterior porous boundary $\partial\Omega_D \setminus \Gamma$. Nondimensional parameters are fixed as $T=1.0$, $\Delta t=0.2$, and $\alpha=\nu=g=S_0=k=1$.
\begin{figure}[H]
	\centering
	\begin{minipage}[c]{0.45\linewidth}
		\centering
		\includegraphics[width=0.6\linewidth]{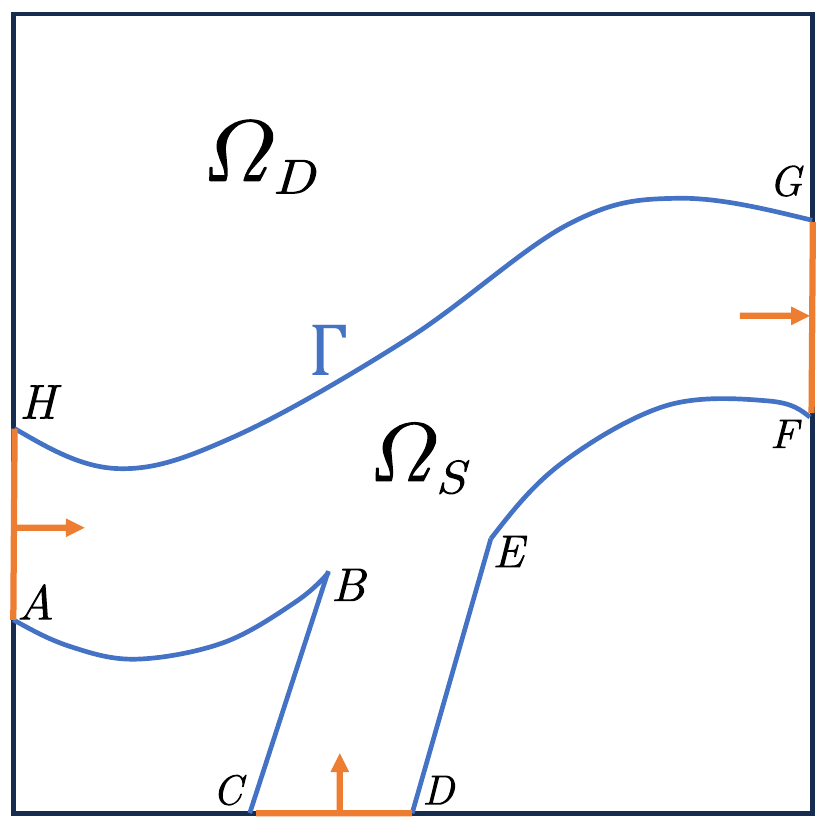}
		\caption{A sketch of flow in a Karst aquifer with curved interfaces for Example 3.}
		\label{ex3-domain}
	\end{minipage}
	\hfill 
	\begin{minipage}[c]{0.52\linewidth}
		\begin{equation*}
			\mathbf{u}_S =
			\begin{cases}
				(\omega_1, 0), & \text{on } \overline{HA} \quad (\text{Inlet 1}),\\
				(0, \omega_1), & \text{on } \overline{CD} \quad (\text{Inlet 2}),\\
				(\omega_2, 0), & \text{on } \overline{FG} \quad (\text{Outlet}).
			\end{cases}
		\end{equation*}
	\end{minipage}
    \vspace{-2em}
\end{figure}

\textbf{Dataset and Configuration:} Reference solutions are computed via a high-fidelity finite element method (FEM) solver using $P_2$-$P_1$ elements for $\mathbf{u}_S, p$ and $P_2$ for $\phi_D$. The training set utilizes sparse temporal snapshots at $t \in \{0.2, 0.4, 0.6, 0.8\}$, reserving $t=1.0$ to rigorously test extrapolation. To isolate the impact of geometric complexity, the ViT-K architecture remains strictly identical to the configuration in Section~\ref{subsec:ex1_stokes_darcy}.

\textbf{Results and Analysis:} Figure \ref{fig:karst} compares the ViT-K predictions and the FEM reference at t = 1.0 s. Within the fracture domain ($\Omega_S$), the accurate resolution of high-velocity jets and flow redirection confirms that the model effectively encodes complex geometric constraints via the region mask. Crucially, streamlines transition smoothly across the interface $\Gamma$, indicating that ViT-K implicitly learns the Beavers-Joseph condition without explicit physics-informed losses, while simultaneously resolving the low-velocity seepage in the porous matrix ($\Omega_D$). Table \ref{tab:karst_error} summarizes the quantitative metrics at t = 1.0 s. Although relative errors are slightly elevated compared to idealized rectangular domains, deviations localize primarily near the irregular boundaries. Overall, ViT-K successfully handles non-smooth heterogeneous geometries and preserves physical consistency during temporal extrapolation.
\begin{table}[htbp]
	\footnotesize
	\centering
	\caption{Quantitative error metrics for Example \ref{subsec:ex3_karst} at extrapolation time $t=1.0\,\mathrm{s}$.}
	\label{tab:karst_error}
	\setlength{\tabcolsep}{5pt}
	\begin{tabular}{lcccc}
		\toprule
		Channel & MSE & RMSE & MAE & Rel. $L^2$ Error  \\
		\midrule
		$u_1$ & $2.32 \times 10^{-3}$ & $4.81 \times 10^{-2}$ & $3.63 \times 10^{-2}$ & $9.70\%$  \\
		$u_2$ & $6.71 \times 10^{-4}$ & $2.59 \times 10^{-2}$ & $2.05 \times 10^{-2}$ & $8.17\%$  \\
		\bottomrule
	\end{tabular}
	\vspace{-1em}
\end{table}
\begin{figure}
	\centering
	\includegraphics[width=0.8\linewidth]{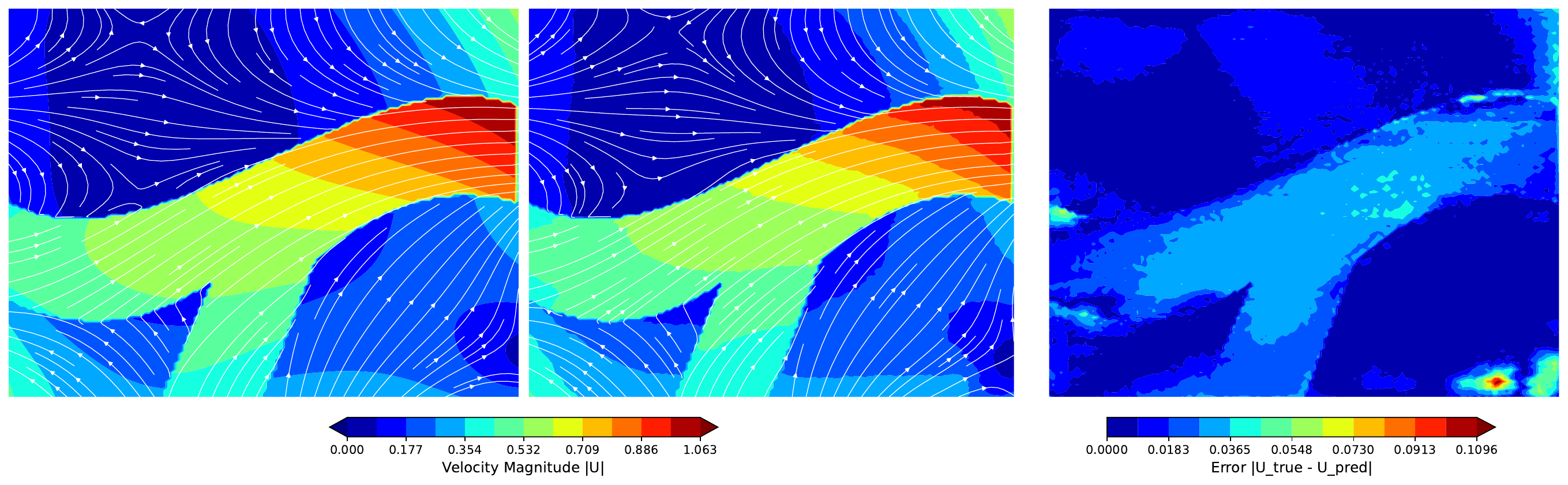}
	\caption{Predictive performance on the Karst aquifer model at extrapolation time $t=1.0\,\mathrm{s}$. Panels display the Ground Truth (left), ViT-K Prediction (middle), and Absolute Error (right) for the velocity magnitude.}
	\label{fig:karst}
	\vspace{-2em}
\end{figure}

\subsection{Example 4: Pulsatile Hemodynamics in Bifurcating Vessels}
\label{subsec:ex4_hemodynamics}
To evaluate ViT-K's capacity to resolve high-frequency pulsatile flows in heterogeneous media, we simulate a bifurcating vessel network embedded in porous biological tissue. The domain $\Omega = [0, 1]^2$ comprises a fluid lumen ($\Omega_S$) branching from a common inlet along cubic centerlines $y_{c1,2}(x) = 0.8 - \{0.4, 1.4\}(x+0.15)^3$, surrounded by a porous tissue matrix ($\Omega_D$). Asymmetric vessel widths ($w_1=0.145, w_2=0.15$) induce flow heterogeneity. Physical parameters reflect typical physiological conditions: density $\rho=1.0$, dynamic viscosity $\mu=0.035$, tissue storage $S_0=0.01$, and permeability $K=10^{-6}$.

\textbf{Governing Equations and Boundary Conditions:} The Navier--Stokes equations (in $\Omega_S$) and Darcy's law (in $\Omega_D$) are coupled via a unified weak form:
\begin{equation}
	F(\boldsymbol{w}; \boldsymbol{v}) = \int_{\Omega_S} \mathcal{R}_{\text{NS}}(\boldsymbol{w}, \boldsymbol{v}) \, \mathrm{d}\boldsymbol{x} + \int_{\Omega_D} \mathcal{R}_{\text{Darcy}}(\boldsymbol{w}, \boldsymbol{v}) \, \mathrm{d}\boldsymbol{x} = 0,
\end{equation}
implicitly enforcing kinematic continuity and normal stress balance across the interface. $\boldsymbol{w}=(\boldsymbol{u}, p)$.  A time-dependent pulsatile velocity profile is imposed at the inlet ($x=0$):
\begin{equation}
	\boldsymbol{u}_{\text{in}}(0, y, t) = U_{\max} \left( 1 - \left(\frac{y-y_c}{w_{\text{in}}}\right)^2 \right) \cdot \alpha(t) \cdot \left( 1 + 0.4\sin(2\pi f t) \right) \boldsymbol{e}_x,
\end{equation}
where peak velocity $U{\max}=1.8$, inlet width $w_{\text{in}}=0.12$, and frequency $f=2.5\,\mathrm{Hz}$ (period $T=0.4\,\mathrm{s}$). To mitigate initial numerical shocks, a cosine ramp-up function $\alpha(t)$ smoothly scales the inflow from $0$ to $1$ during the first cardiac cycle ($t \le 0.4\,\mathrm{s}$). Outlets ($x=1, y=0$) are set to stress-free conditions ($p=0$), with no-slip boundaries on the remaining walls.

\textbf{Dataset Generation:} High-fidelity reference data is generated using a monolithic mixed FEM solver (FEniCS) on a $128 \times 128$ grid, employing $P_2$--$P_1$ Taylor--Hood elements and implicit Euler time integration ($\Delta t = 0.02\,\mathrm{s}$). To eliminate transients, the initial warm-up cycle ($t < 0.4\,\mathrm{s}$) is discarded. The model is trained on the subsequent 2.5 cycles ($t \in [0.4, 1.4]\,\mathrm{s}$) and tested on the following 2.5 cycles ($t \in [1.4, 2.4]\,\mathrm{s}$). This temporal extrapolation strategy rigorously verifies that ViT-K learns the underlying periodic dynamics rather than overfitting a specific trajectory.

\textbf{Non-Autonomous ViT-K Implementation:} To accommodate the strong external pulsatile forcing, we extend our framework to a non-autonomous Koopman formulation:
\begin{equation}
	\boldsymbol{z}(t) = \mathcal{K}_{\text{nat}}(\boldsymbol{z}_0, t) + \mathcal{K}_{\text{force}}(\boldsymbol{u}_{\text{ext}}(t)).
\end{equation}
The natural dynamics $\mathcal{K}_{\text{nat}}$ is governed by a block-diagonal generator with stable eigenvalues ($\mathrm{Re}(\lambda) < 0$), capturing inherent system relaxation. The forced response $\mathcal{K}_{\text{force}}$ injects cardiac cycle features $[\sin(2\pi f t), \cos(2\pi f t)]$ via a two-layer MLP, allowing synchronization with the driving frequency. The ViT encoder utilizes an $8 \times 8$ patch size, an embedding dimension of 192, and 6 layers. Training spans 2000 epochs with a composite loss balancing velocity MSE ($w_u=10$), robust Huber loss for pressure ($w_p=20$), global pressure mean constraint ($w_{\text{bias}}=100$), and an $H^1$ spatial gradient penalty ($w_{\text{grad}}=5$).

\textbf{Results and Discussion:} Figures \ref{fig:vessel_comp_t1.5} and \ref{fig:vessel_flow_evolution} evaluate ViT-K against the FEM reference during the extrapolation phase ($t \ge 1.5\,\mathrm{s}$). The model accurately resolves complex pulsatile hemodynamics—including rhythmic vortex modulation, flow separation at the bifurcation, and significant pressure drops—maintaining strict physical consistency up to $t=2.0\,\mathrm{s}$. Unlike autoregressive solvers prone to temporal drift, ViT-K remains phase-locked to the pulsatile driver well beyond the training horizon. This confirms that the non-autonomous Koopman formulation successfully isolates the periodic forced response from transient natural dynamics.

Regarding computational efficiency (Table \ref{tab:time_comparison_custom}), ViT-K achieves a $5.2\times$ acceleration over the high-fidelity FEM baseline ($194\,\mathrm{s}$ vs. $1008\,\mathrm{s}$) for a $5\,\mathrm{s}$ physical simulation. This computational advantage scales linearly for longer horizons, as ViT-K maintains a constant inference cost per time step, effectively bypassing the stringent stability constraints and ill-conditioning issues often encountered in traditional solvers.
\begin{figure}[h]
	\centering
	\includegraphics[width=0.7\linewidth]{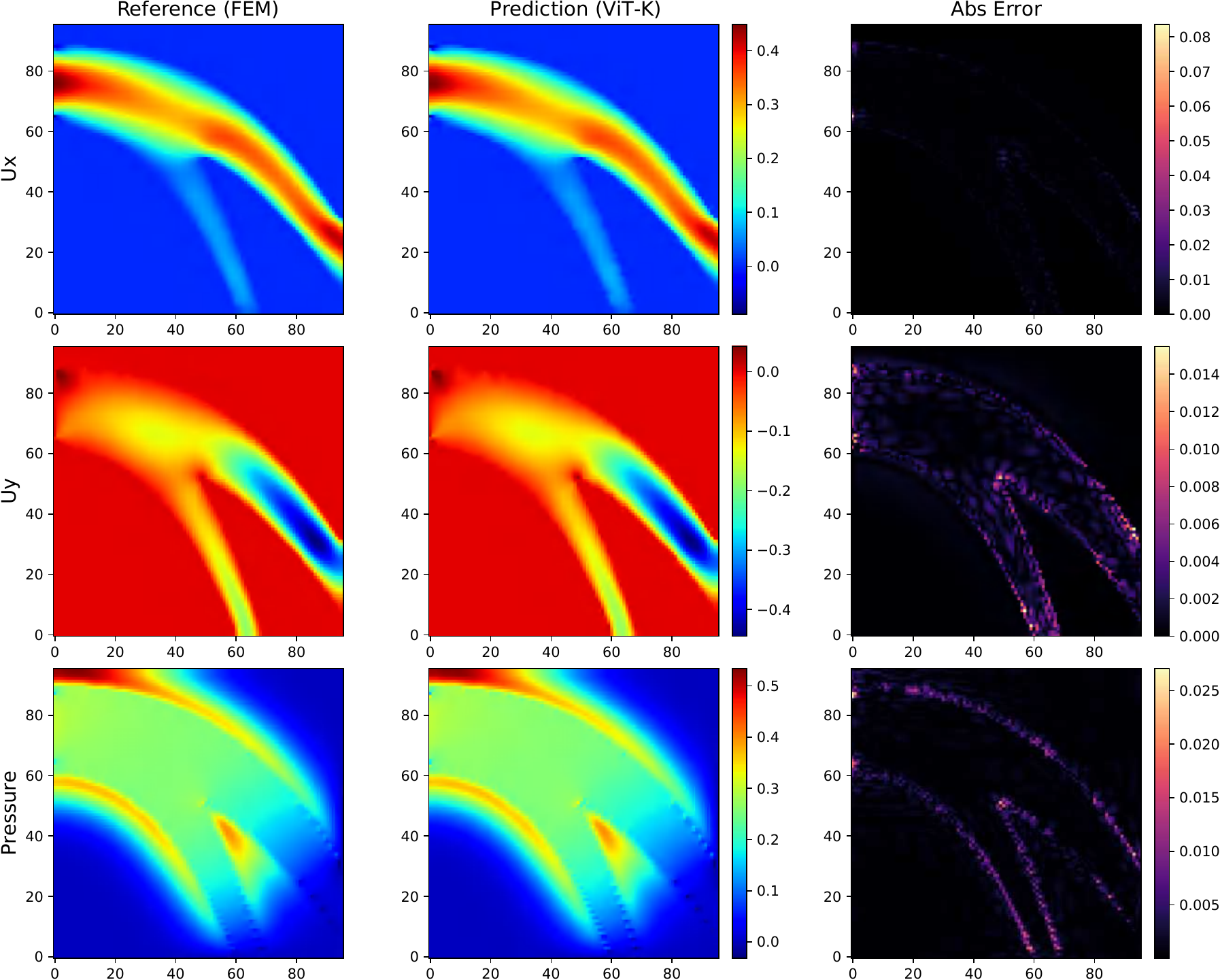}
	\caption{Comparison at $t=1.5\,\mathrm{s}$ (Extrapolation Phase). Long-term stability is maintained due to the non-autonomous forcing formulation.}
	\label{fig:vessel_comp_t1.5}
	\vspace{-2em}
\end{figure}
\begin{figure}[h]
	\centering
	\begin{subfigure}[b]{0.4\textwidth}
		\centering
		\includegraphics[width=\linewidth]{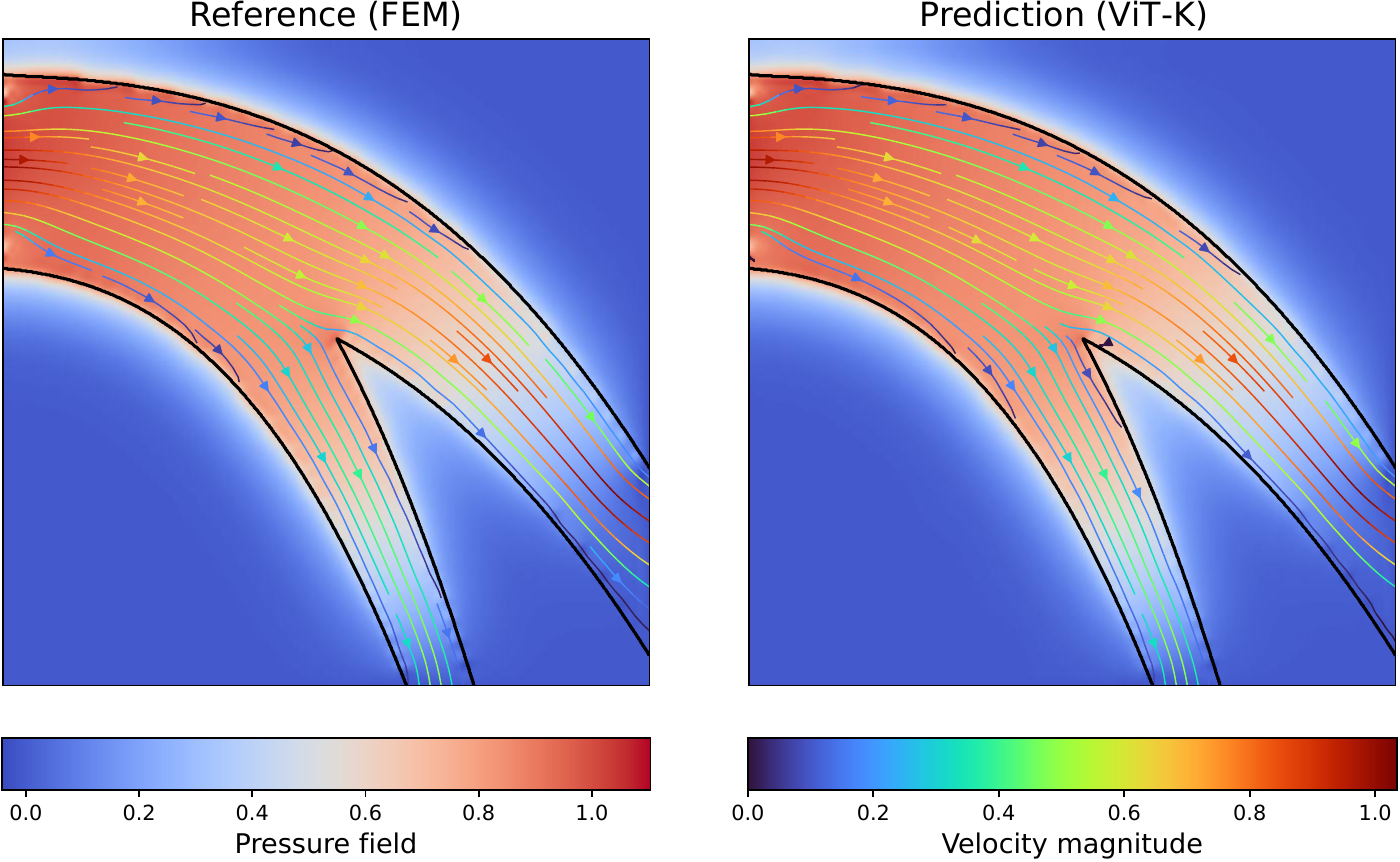}
		\caption{$t=0.5\,\mathrm{s}$}  
		\label{fig:vessel_flow_evolution0.5}
	\end{subfigure}
	\begin{subfigure}[b]{0.4\textwidth}
		\centering
		\includegraphics[width=\linewidth]{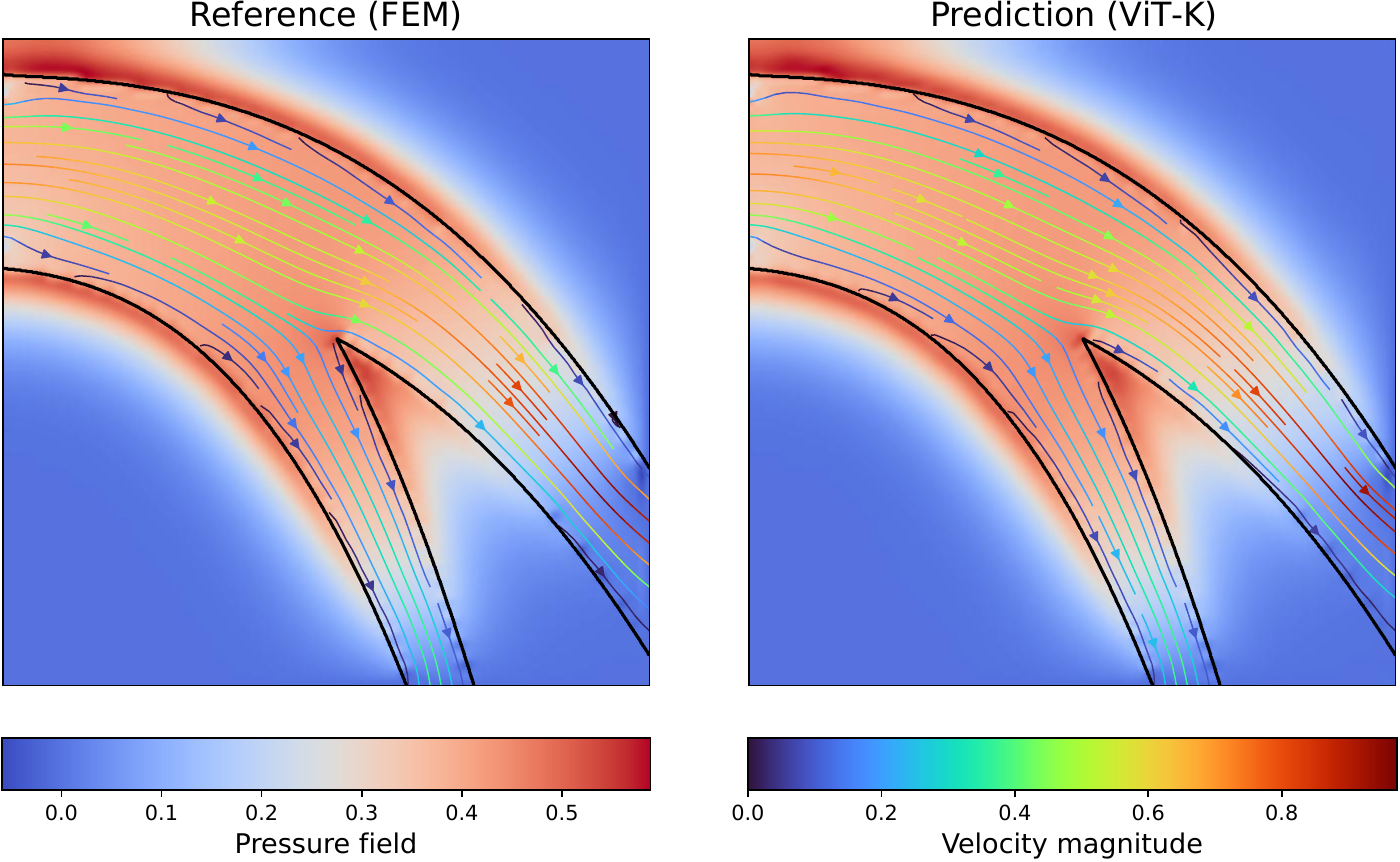}
		\caption{$t=1.0\,\mathrm{s}$}
		\label{fig:vessel_flow_evolution1.0}
	\end{subfigure}
	\vspace{0.2em}  
	\begin{subfigure}[b]{0.4\textwidth}
		\centering
		\includegraphics[width=\linewidth]{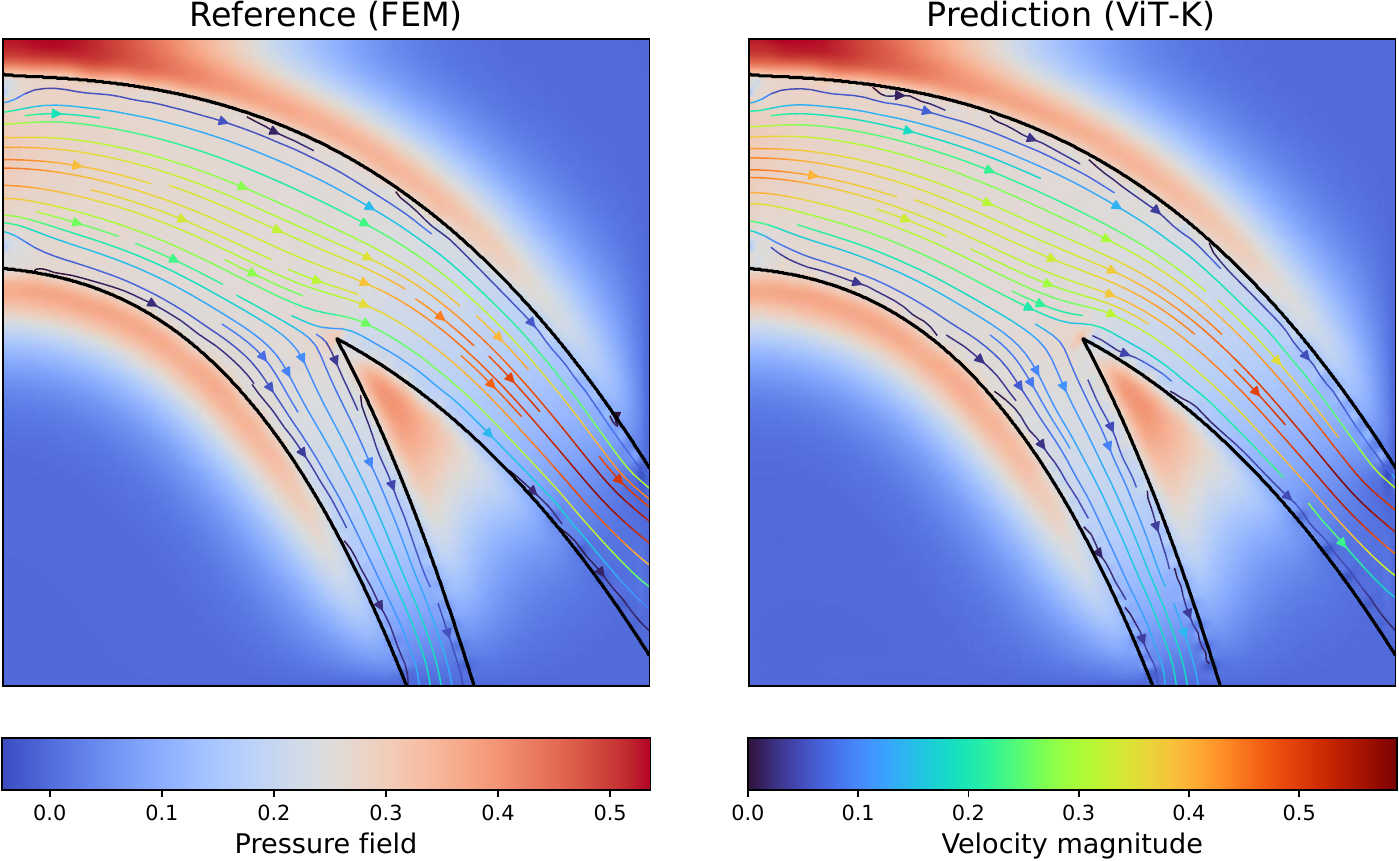}
		\caption{$t=1.5\,\mathrm{s}$}
		\label{fig:vessel_flow_evolution1.5}
	\end{subfigure}
	\begin{subfigure}[b]{0.4\textwidth}
		\centering
		\includegraphics[width=\linewidth]{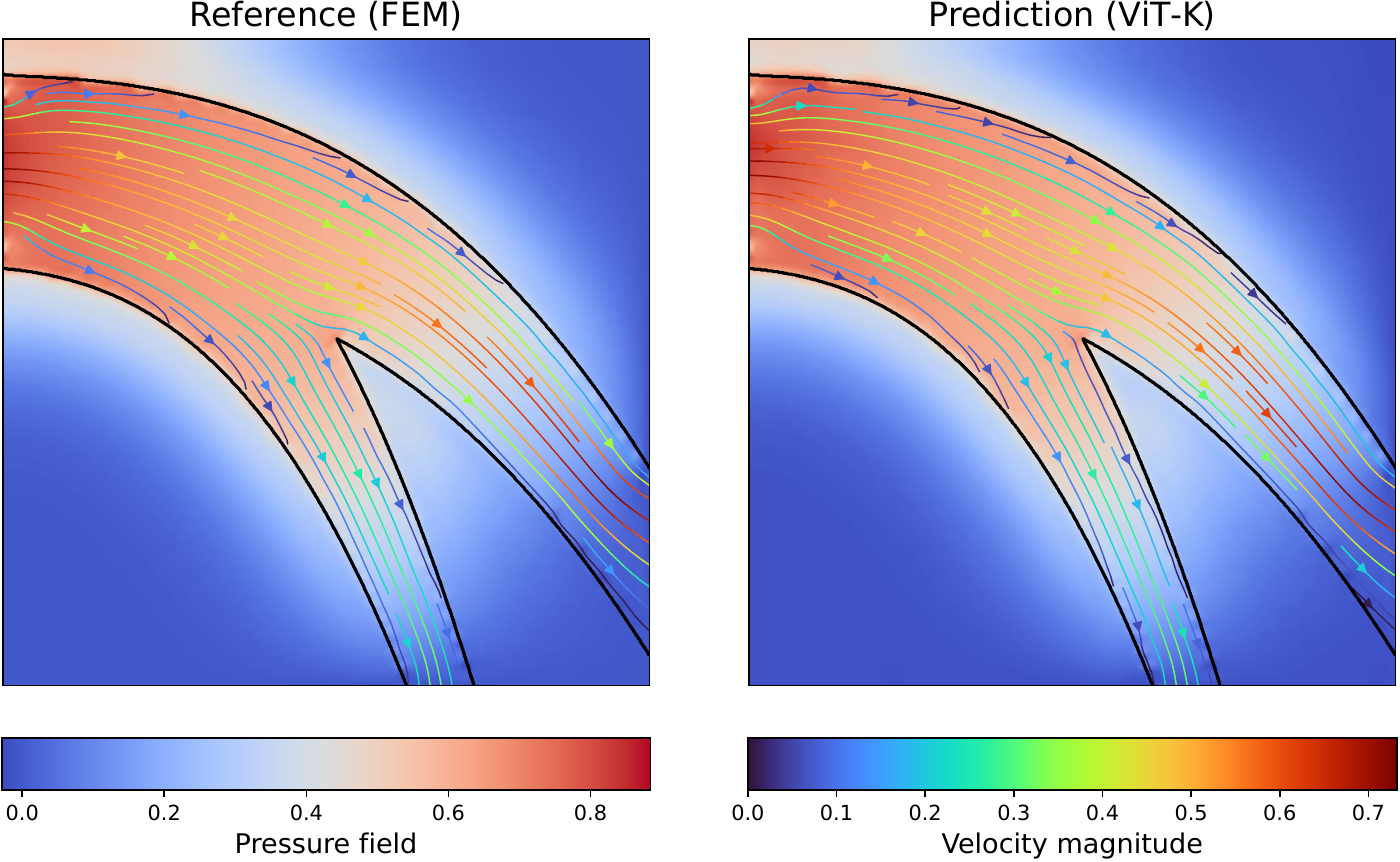}
		\caption{$t=2.0\,\mathrm{s}$}
		\label{fig:vessel_flow_evolution2.0}
	\end{subfigure}
	\vspace{-2em}
	\caption{Evolution of flow field streamlines and pressure over multiple cardiac cycles. The model captures the phase-dependent variations in vortex structure.}
	\label{fig:vessel_flow_evolution}
	\vspace{-3em}
\end{figure}
{\footnotesize
\begin{table}
	\centering
	\caption{Computational Efficiency (FEM vs. ViT-K) for simulating $5\,\mathrm{s}$ of physical time.}
	\renewcommand{\arraystretch}{1.2}
	\setlength{\tabcolsep}{6pt}
	\begin{tabular}{lcc}
		\toprule
		Method & Wall-Clock Time (s) & Speed-up Factor \\
		\midrule
		FEM & 1008.23 & $1.0\times$ \\
		ViT-K & 193.88 & $\mathbf{5.2\times}$ \\
		\bottomrule
	\end{tabular}
	\parbox{0.6\textwidth}{\footnotesize \textit{Note:} Evaluation performed on the same hardware. ViT-K time includes data training.}
	
	\label{tab:time_comparison_custom}
	\vspace{-2em}
\end{table}}

\subsection{Few-Shot Learning and Long-Term Extrapolation Capabilities}
\label{sec:long_term_stability}

A distinguishing feature of the proposed ViT-K framework is its data efficiency and asymptotic stability. Unlike purely autoregressive deep learning models (e.g., LSTMs or standard RNNs) which often suffer from exponential error divergence, our structured Koopman formulation guarantees bounded error growth.
In this section, we rigorously validate these properties through extreme extrapolation tests on the periodic Navier-Stokes-Darcy system (Example 2) and the pulsatile vessel flow (Example 4).

\subsubsection{Asymptotic Stability in Periodic Flows (Example 2)}
To rigorously assess the extrapolation limits of our framework, we conduct a few-shot learning experiment on the Navier--Stokes--Darcy system (Section \ref{subsec:ex2_nsd}). Trained on a sparse dataset of merely 10 snapshots ($t \in [0, 1.0]$, $\Delta t=0.1$), the model is tasked to predict the system dynamics up to $t=100.0$, effectively achieving an extreme $100\times$ extrapolation beyond the training horizon.

Figure \ref{fig:nsd-long-100} tracks the Root Mean Squared Error (RMSE) evolution over the extreme extrapolation horizon up to $t=100.0\,\mathrm{s}$. Crucially, the phase-averaged error follows a strictly linear trajectory with a slope of $10^{-4}$. This empirical result perfectly corroborates the theoretical bound derived in Theorem \ref{thm:global_error}, confirming that the stable Koopman formulation successfully suppresses exponential divergence even when extrapolating from minimal training data.
\begin{figure}[h]
	\centering
	\includegraphics[width=0.8\linewidth]{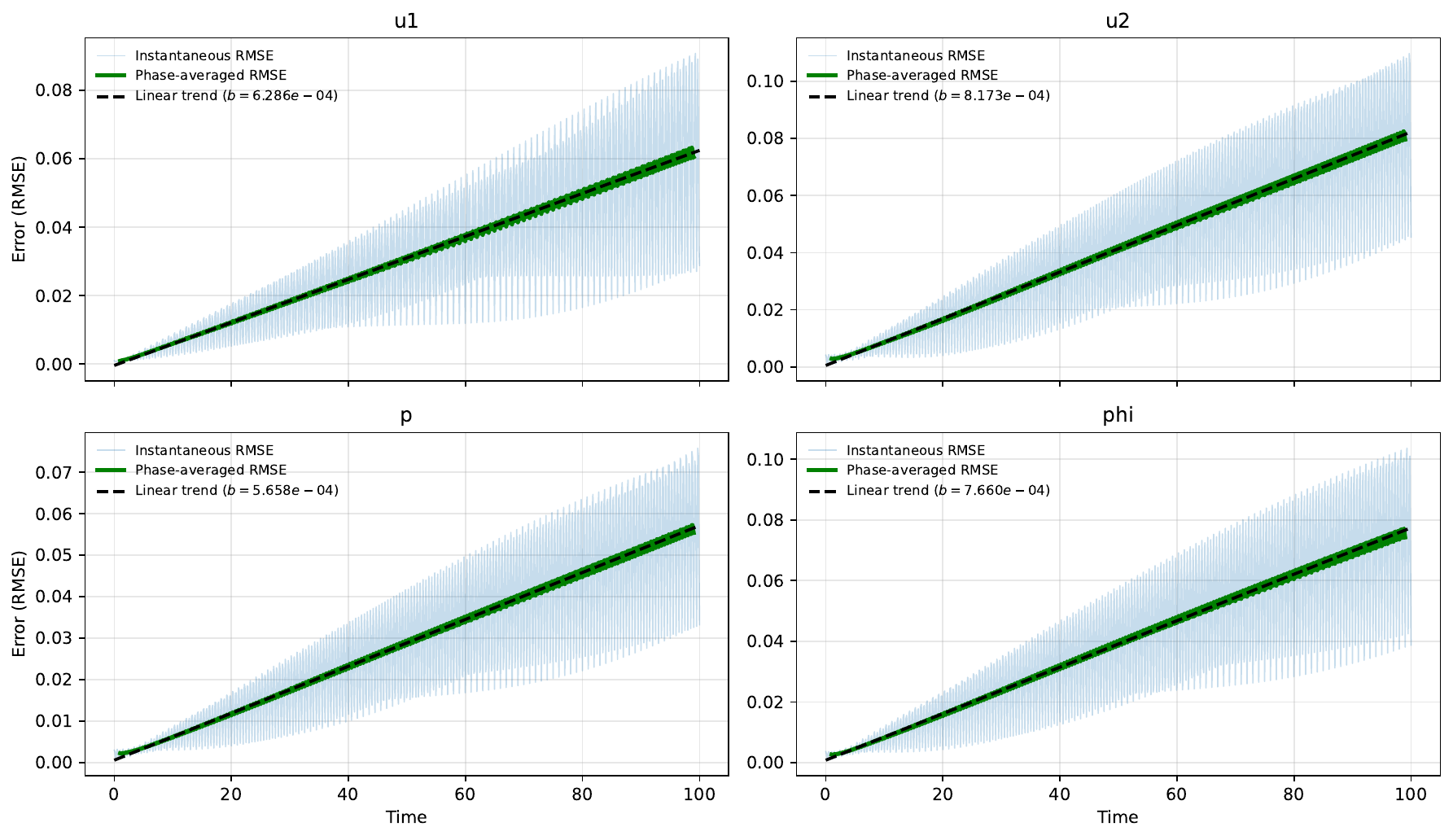}
	\caption{Long-term RMSE evolution up to $t=100\,\mathrm{s}$. The oscillatory error bounds closely track a linear growth trend, empirically validating the stability guarantees of Theorem \ref{thm:global_error}.}
	\label{fig:nsd-long-100}\vspace{-2em}
\end{figure}

Table \ref{tab:nsd_time_metrics} quantitatively substantiates this long-term stability by comparing predictions at $t=10.0\,\mathrm{s}$ and $t=100.0\,\mathrm{s}$ horizons. Notably, when extending the prediction horizon by a factor of 10, the relative $L^2$ errors merely double (e.g., from $2.07\%$ to $3.47\%$ for $u_1$, and $2.76\%$ to $5.22\%$ for $\phi$), strictly avoiding exponential amplification. Furthermore, the maximum pointwise errors remain remarkably stable (bounded around $7 \times 10^{-2}$) across all variables. This confirms that while minor projection errors ($\epsilon_{proj}$) accumulate over hundreds of recursive steps, they do not induce system blow-up, ensuring globally stable and reliable macroscale reconstructions.
\begin{table}
	\footnotesize
	\centering
	\caption{Quantitative comparison of prediction accuracy at $t=10\mathrm{s}$ and $t=100\mathrm{s}$}
	\label{tab:nsd_time_metrics}
	\setlength{\tabcolsep}{5pt}
	\begin{tabular}{lccccc}
		\toprule
		Time & Channel & MSE & MAE & Max Error & Rel-L2 (\%) \\
		\midrule
		\multirow{4}{*}{$t=10\mathrm{s}$}
		& $u_1$  & $4.21 \times 10^{-5}$ & $4.51 \times 10^{-3}$ & $6.23 \times 10^{-2}$ & $2.07$ \\
		& $u_2$  & $5.94 \times 10^{-5}$ & $5.74 \times 10^{-3}$ & $5.03 \times 10^{-2}$ & $2.23$ \\
		& $p$    & $3.85 \times 10^{-5}$ & $4.47 \times 10^{-3}$ & $4.19 \times 10^{-2}$ & $2.56$ \\
		& $\phi$ & $4.75 \times 10^{-5}$ & $5.01 \times 10^{-3}$ & $5.89 \times 10^{-2}$ & $2.76$ \\
		\midrule
		\multirow{4}{*}{$t=100\mathrm{s}$}
		& $u_1$  & $7.88 \times 10^{-5}$ & $6.59 \times 10^{-3}$ & $6.97 \times 10^{-2}$ & $3.47$ \\
		& $u_2$  & $1.93 \times 10^{-4}$ & $1.04 \times 10^{-2}$ & $7.02 \times 10^{-2}$ & $4.01$ \\
		& $p$    & $1.12 \times 10^{-4}$ & $7.73 \times 10^{-3}$ & $5.15 \times 10^{-2}$ & $4.37$ \\
		& $\phi$ & $3.70 \times 10^{-4}$ & $9.88 \times 10^{-3}$ & $5.76 \times 10^{-2}$ & $5.22$ \\
		\bottomrule
	\end{tabular}\vspace{-2em}
\end{table}
%

\subsubsection{Phase Consistency in Pulsatile Flows (Example 4)}
To rigorously assess robustness in non-autonomous hemodynamics, we extrapolate the branching vessel dynamics far beyond the $1.0\,\mathrm{s}$ (2.5 cycles) training window. Specifically, we extend the prediction horizon to $t=5.0\,\mathrm{s}$ and $t=50.0\,\mathrm{s}$, corresponding to 12.5 and 125 full cardiac cycles, respectively.

As shown in Figure \ref{fig:vessel_flow_longterm}, ViT-K exhibits exceptional phase-locking capability. Despite the complex fluid--porous heterogeneity, the predicted flow avoids phase drift entirely over extended horizons. Flow diversion ratios and vortex shedding frequencies remain physically consistent with the reference hemodynamics. These results validate the efficacy of the non-autonomous Koopman formulation in strictly capturing the invariant spectral properties of the cardiovascular attractor.

\begin{figure}[h]
	\centering
	\begin{subfigure}[b]{0.48\linewidth}
		\centering
		\includegraphics[width=\linewidth]{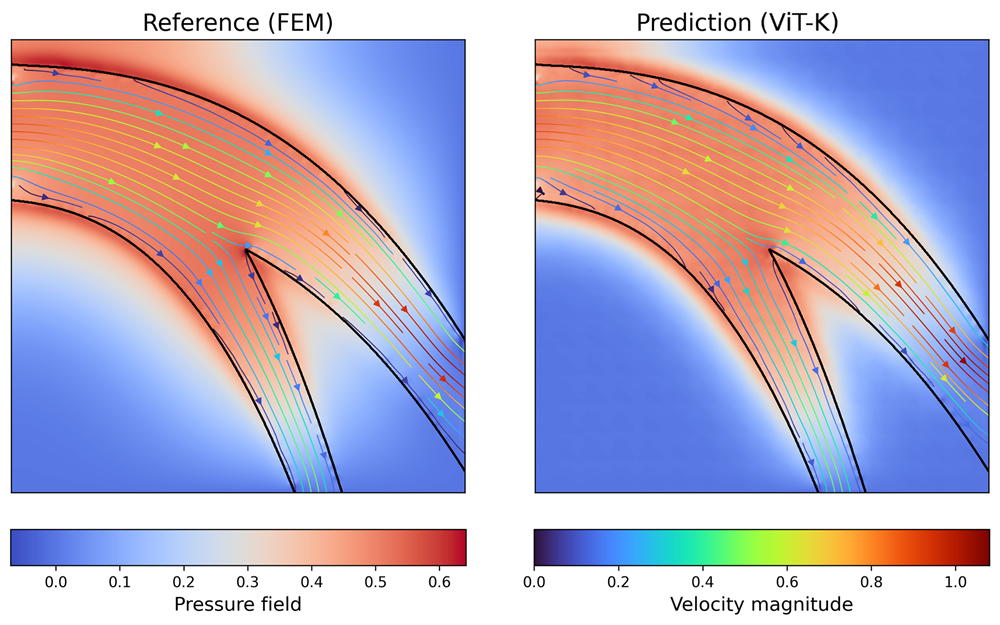}
		\caption{Flow visualization at \( t = 5.0 \) (12.5th cardiac cycle).}
		\label{fig:vessel_flow_evolution5.0}
	\end{subfigure}
	\hfill
	\begin{subfigure}[b]{0.48\linewidth}
		\centering
		\includegraphics[width=\linewidth]{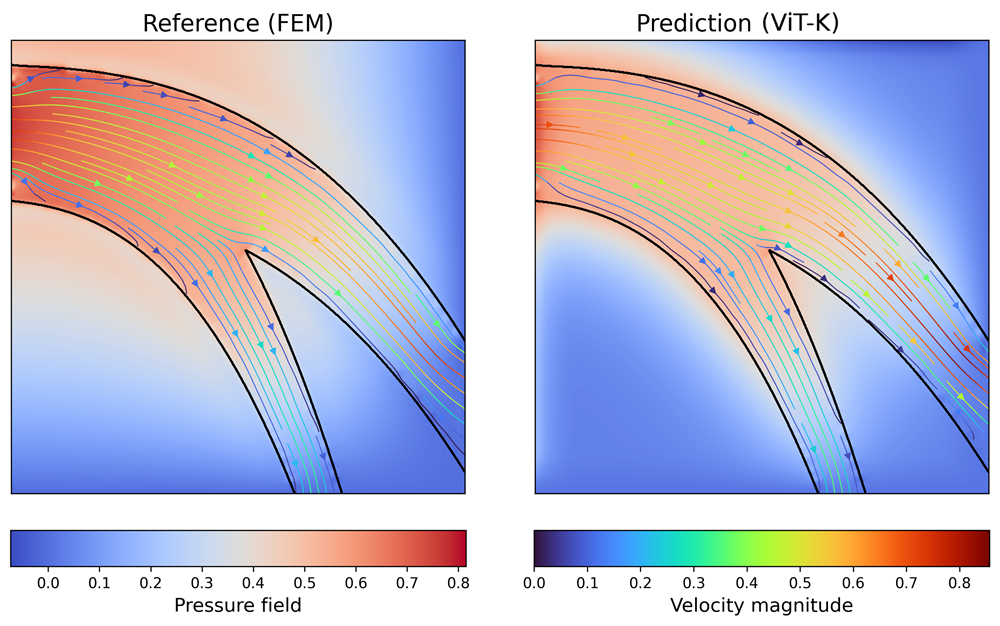}
		\caption{Flow visualization at \( t = 50.0 \) (125th cardiac cycle).}
		\label{fig:vessel_flow_evolution50.0}
	\end{subfigure}\vspace{-1em}
	\caption{Long-term flow predictions for the bifurcating vessel. ViT-K preserves pulsatile dynamics and strict phase consistency without non-physical dissipation over extreme extrapolation horizons.}
	\label{fig:vessel_flow_longterm}\vspace{-3em}
\end{figure}

\subsection{Robustness Analysis under Measurement Noise}
To evaluate ViT-K's robustness against measurement noise, which is a ubiquitous challenge in clinical hemodynamics, we perturb the clean validation data $\mathbf{u}_{\text{ref}}$ with additive white Gaussian noise (AWGN) $\tilde{\mathbf{u}}(x, y, t) = \mathbf{u}_{\text{ref}}(x, y, t) + \delta \cdot \|\mathbf{u}_{\text{ref}}\|_\infty \cdot \xi, \quad \xi \sim \mathcal{N}(0, 1),$ where $\|\mathbf{u}_{\text{ref}}\|_\infty$ denotes the maximum field amplitude, $\delta \in \{0.1, 0.15\}$ imposes $10\%$ and $15\%$ noise intensities, and $\xi$ follows a standard normal distribution. Such high-frequency stochastic oscillations severely compromise local gradient regularity, posing significant challenges for solvers reliant on well-posed, smooth initial conditions.
\begin{figure}[h]
	\centering
	\begin{subfigure}[b]{0.49\linewidth}
		\centering
		\includegraphics[width=\linewidth]{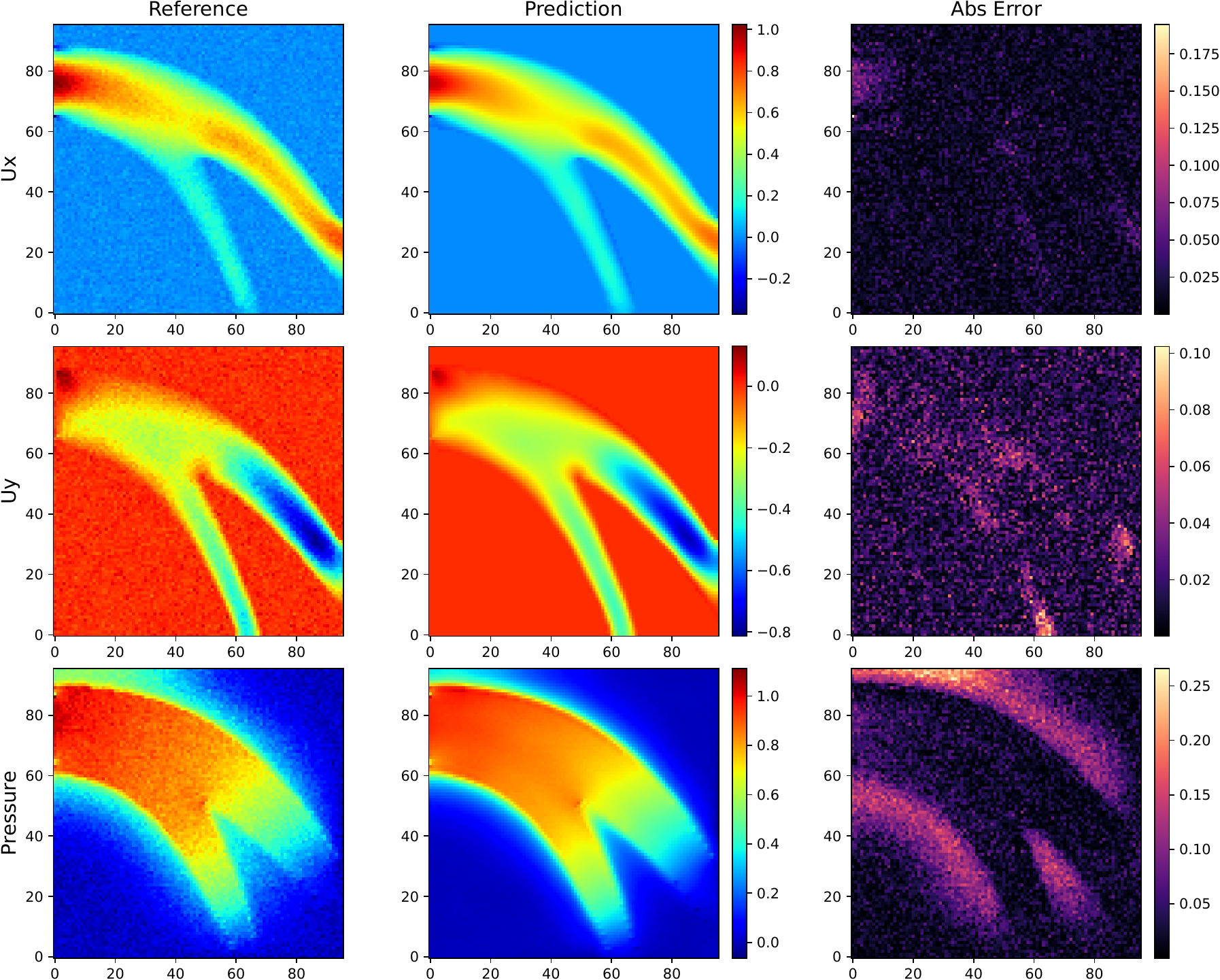}
		\caption{$10\%$ Gaussian noise}
		\label{fig:robustness_10}
	\end{subfigure}
	\hfill
	\begin{subfigure}[b]{0.49\linewidth}
		\centering
		\includegraphics[width=\linewidth]{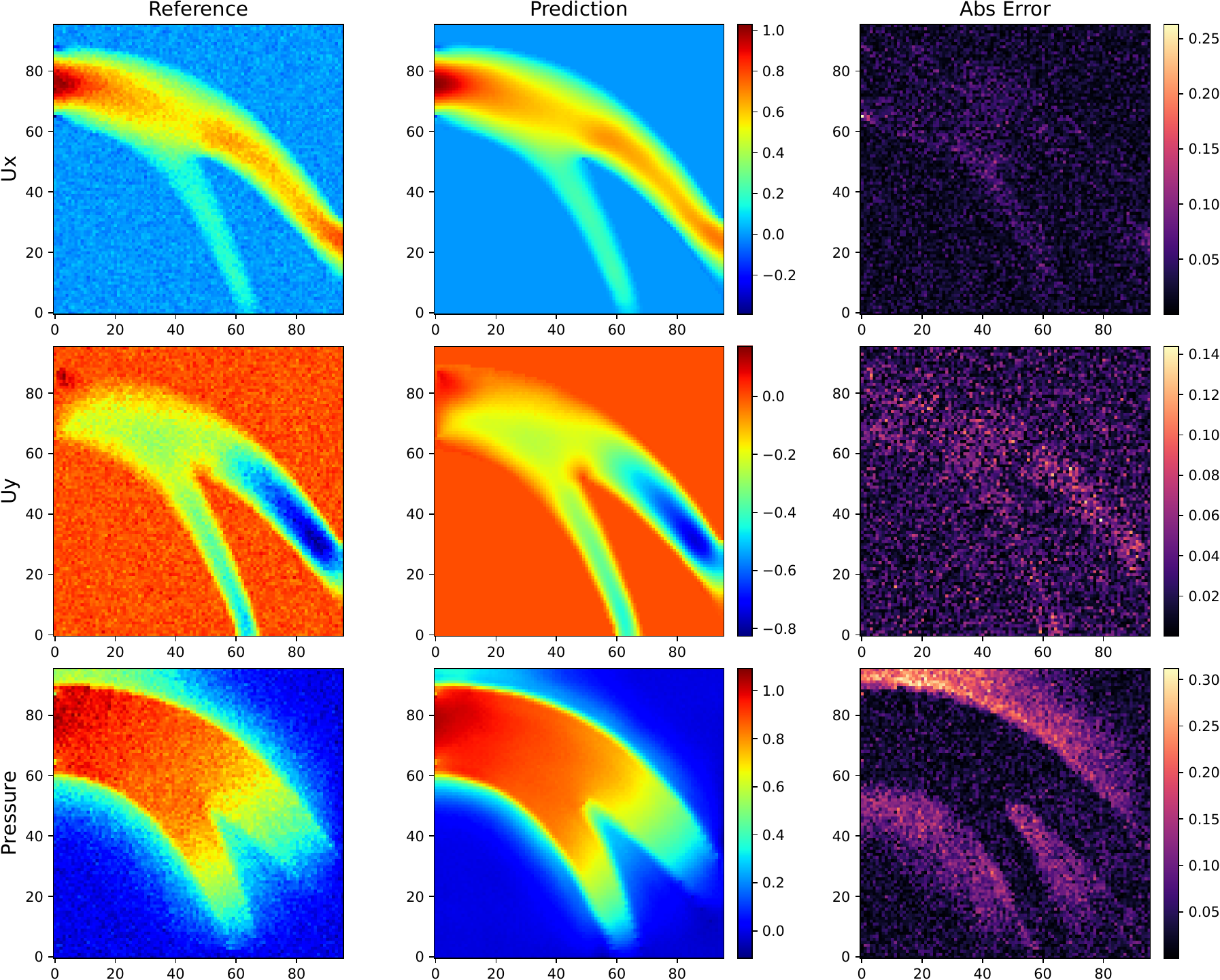}
		\caption{$15\%$ Gaussian noise}
		\label{fig:robustness_15}
	\end{subfigure}\vspace{-1.5em}
	\caption{Comparison of the FEM reference with Gaussian noise (left column), the ViT-K approximation (middle column), and the absolute error (right column) at $t=2.5\,\mathrm{s}$ under different noise levels (testing set).}
	\label{fig:robustness_noise}\vspace{-2em}
\end{figure}

\begin{figure}[h]
	\centering
	\includegraphics[width=0.80\linewidth]{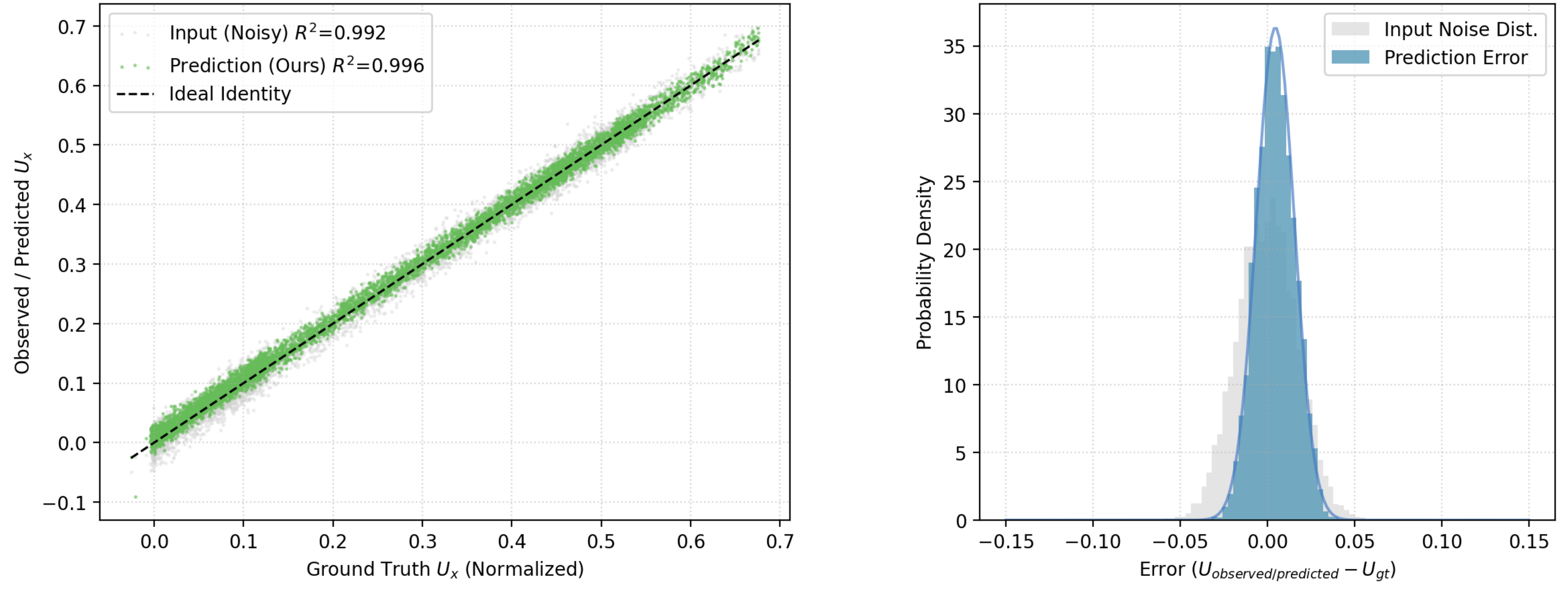}
	\caption{Prediction performance of physical quantity $\mathbf{u}$ (left) and probability density distribution of prediction errors (right) under noisy input.}
	\label{fig:robustness3}\vspace{-3em}
\end{figure}
Figures \ref{fig:robustness_noise} and \ref{fig:robustness3} demonstrate the model's robust denoising capacity. Qualitatively, although the input fields $\tilde{\mathbf{u}}$ exhibit severe non-physical fluctuations, the ViT-K predictions at both noise levels achieve remarkable physical reconstruction, yielding smooth fields structurally consistent with the ground truth (Fig. \ref{fig:robustness_noise}). Quantitatively (Fig. \ref{fig:robustness3}), ViT-K narrows the pointwise scatter (improving the coefficient of determination from $R^2=0.992$ to $0.996$) and transforms the broad input noise into a sharp, leptokurtic error distribution centered at zero, confirming substantial variance reduction.

This intrinsic regularization is governed by the spectral properties of the learned Koopman operator. Because high-frequency AWGN is largely orthogonal to the dominant eigenfunctions defining the system's low-dimensional inertial manifold, ViT-K acts as an implicit spectral filter. It projects the noisy state onto the learned manifold of valid PDE solutions, effectively attenuating non-physical stochastic components rather than overfitting to pointwise corrupted data.

\section{Conclusion}
\label{sec:conclusion}
In this work, we present ViT-K, a few-shot framework that synergizes Vision Transformers (ViT) with Koopman operator theory to simulate coupled fluid–porous media flows. By explicitly enforcing stability constraints within the Koopman generator, ViT-K overcomes the exponential error accumulation typical of recurrent networks, theoretically and numerically guaranteeing linear error growth. This critical property enables reliable temporal extrapolation up to $100\times$ the training horizon using only sparse data. Spatially, the ViT encoder leverages self-attention to accurately resolve heterogeneous Beavers–Joseph interface conditions across complex geometries (e.g., Karst aquifers, bifurcating vessels) while demonstrating intrinsic robustness against measurement noise. Future research will extend ViT-K to three-dimensional irregular geometries via Graph Neural Networks and address high-Reynolds-number flows requiring adaptive spectral pruning for continuous spectra. Ultimately, ViT-K offers a robust, interpretable paradigm that effectively bridges data-driven computational efficiency with physical reliability for multiphysics modeling.

\bibliographystyle{siamplain}
\bibliography{ViT-K-refs}
\end{document}

%% file: ex_shared.tex
\usepackage{lipsum}
\usepackage{amsfonts}
\usepackage{graphicx}
\usepackage{epstopdf}
\usepackage{algorithmic}
\ifpdf
  \DeclareGraphicsExtensions{.eps,.pdf,.png,.jpg}
\else
  \DeclareGraphicsExtensions{.eps}
\fi

\newsiamremark{remark}{Remark}
\newsiamremark{hypothesis}{Hypothesis}
\crefname{hypothesis}{Hypothesis}{Hypotheses}
\newsiamthm{claim}{Claim}
\newsiamremark{fact}{Fact}
\crefname{fact}{Fact}{Facts}

\headers{ViT-K}{M. Chen, C. Qiu, Z. Mao and M. Xu}

\title{ViT-K: A Few-Shot Learning Model for Coupled Fluid-Porous Media Flows with Interface Conditions\thanks{Submitted to the SIAM J. Sci. Comput.
		\funding{Research work was partially supported by Zhejiang Provincial Natural Science Foundation of China under Grant No. LMS26A010009, NSFC Grant No. 12201327 and Natural Science Foundation of Ningbo Municipality Grant No. 2022J087.}}}

\author{Mengjia Chen\thanks{School of Mathematics and Statistics, Ningbo University, Ningbo, Zhejiang 315211, P. R. China 
		(\email{236002309@nbu.edu.cn},\email{qiuchangxin@nbu.edu.cn}).}
	\and Changxin Qiu\footnotemark[2]
	\and Zhiping Mao\thanks{School of Mathematical Sciences, Eastern Institute of Technology, Ningbo, Zhejiang 315200, P. R. China (\email{zmao@eitech.edu.cn}).}
	\and Menghui Xu\thanks{Faculty of Mechanical Engineering and Mechanics, Ningbo University, Ningbo, Zhejiang 315211, P. R. China (xumenghui@nbu.edu.cn).}}

\usepackage{amsopn}
